\documentclass{amsart}
\usepackage{amsfonts,amssymb,stmaryrd,amscd,amsmath,latexsym,amsbsy,mathrsfs}
\numberwithin{equation}{section}
\theoremstyle{plain}
\newtheorem{thm}{Theorem}[section]

\newtheorem{prop}[thm]{Proposition}

\newtheorem{defi}[thm]{Definition}
\newtheorem{lem}[thm]{Lemma}
\newtheorem{cor}[thm]{Corollary}
\newtheorem{eg}[thm]{{Example}}

\theoremstyle{remark}
\newtheorem{rema}[thm]{Remark}
\newtheorem{rems}[thm]{Remarks}

\newcommand{\Ad}{{\mbox{\upshape{Ad}}}}

\newcommand{\C}{{\mathbb C}}
\newcommand{\ch}{\mathrm{ch}}

\newcommand{\N}{{\mathbb N}}

\newcommand{\cB}{{\mathcal B}}

\newcommand{\cF}{{\mathcal F}}

\newcommand{\cH}{{\mathcal H}}
\newcommand{\cI}{{\mathcal I}}

\newcommand{\cJ}{{\mathcal J}}

\newcommand{\cO}{{\mathcal O}}

\newcommand{\cR}{{\mathcal R}}
\newcommand{\cS}{{\mathcal S}}

\newcommand{\deltah}{\hat{\delta}}

\newcommand{\gfrak}{{\mathfrak g}}

\newcommand{\grad}{\mbox{grad}}

\newcommand{\hfrak}{{\mathfrak h}}

\newcommand{\Hom}{{\mathrm{Hom}}}

\newcommand{\id}{{\mbox{id}}}

\newcommand{\Kfrak}{{\mathfrak K}}

\newcommand{\lact}{{\triangleright}}

\newcommand{\Lin}{\mbox{Lin}}

\newcommand{\msA}{\mathscr{A}}

\newcommand{\nfrak}{{\mathfrak n}}

\newcommand{\lfrak}{\mathfrak{l}}

\newcommand{\ochi}{\overline{\chi}}

\newcommand{\oP}{{\overline{P}}}

\newcommand{\oV}{{\overline{V}}}

\newcommand{\on}{\overline{n}}

\newcommand{\ot}{\otimes}
\newcommand{\otau}{\overline{\tau}}

\newcommand{\ract}{\triangleleft}

\newcommand{\rt}{\mathrm{rt}}

\newcommand{\slfrak}{{\mathfrak{sl}}}
\newcommand{\slzh}{\widehat{\mathfrak{sl}}_2}
\newcommand{\slzhth}{\widehat{\mathfrak{sl}}{}_2^{\thetah}}
\newcommand{\slzhm}{{\widehat{\mathfrak{sl}}{}^-_2}}
\newcommand{\slzht}{{\widehat{\mathfrak{sl}}{}^\theta_2}}
\newcommand{\SLzh}{\widehat{SL}_2}

\newcommand{\tB}{\tilde{B}}

\newcommand{\thetah}{{\widehat{\theta}}}

\newcommand{\vep}{\varepsilon}

\newcommand{\Z}{{\mathbb Z}}

\title{Radial part calculations for $\slzh$ and the Heun KZB-heat equation}
\author{Stefan Kolb}
\address{School of Mathematics and Statistics, Newcastle University,
   Herschel Building, Newcastle upon Tyne NE1 7RU, UK} 
\email{stefan.kolb@newcastle.ac.uk}

\subjclass[2010]{17B67, 17B80, 81R10}
\keywords{Kac-Moody algebras, symmetric Lie algebras, radial part, Calogero-Moser system, KZB-heat equation, symmetric theta functions}

\thanks{The early stages of this work were supported by the Netherlands Organization for Scientific Research (NWO) within the VIDI-project ``Symmetry and modularity in exactly solvable models".}

\begin{document}
\begin{abstract}
  In the present paper we determine the radial part of the Casimir element for the affine Lie algebra $\slzh$ with respect to the Chevalley involution. The resulting operator is identified with a blend of the Inozemtsev Hamiltonian and the KZB-heat equation in dimension one. Moreover, it is shown how the corresponding zonal spherical functions give rise to symmetric theta functions and convergence is discussed. The paper takes guidance from previous work by Etingof and Kirillov on the diagonal case. 
\end{abstract}

\maketitle
\section{Introduction}
The Inozemtsev Hamiltonian \cite{a-Ino89} is a generalization of the elliptic Calogero-Moser system for the non-reduced root system of type $BC_N$. It was shown by Ochiai, Oshima, and Sekiguchi to be the universal quantum integrable Hamiltonian with $B_N$-symmetry \cite{a-OOS94}. For $N=1$ one obtains the Hamiltonian
\begin{align}\label{eq:Heun-Darboux}
  H_\ell = \frac{d^2}{dx^2} - \sum_{j=0}^3 l_j(l_j+1) \wp_j(x)
\end{align}
where $\wp_j$ for $j=0,\dots,3$ denote Weierstrass $\wp$-functions shifted by the four half-periods of the underlying lattice, respectively, and $\ell=(l_0,l_1,l_2,l_3)$ are four coupling constants. As shown in \cite[Section 8]{a-OS95} the eigenvalue problem for $H_\ell$ can be reformulated in terms of the Heun equation which is a standard form of a Fuchsian differential equation with four regular singularities, see also \cite[Section 2]{a-Takemura03}. The investigation of the spectral problem for the Inozemtsev Hamiltonian is an active area of ongoing research, even in the Heun case, see for example \cite{a-Ruijsenaars09} or the series of papers by Takemura beginning with \cite{a-Takemura03}.
It should be mentioned, as pointed out in \cite{a-MatSmir06}, \cite{a-Veselov11}, that the operator \eqref{eq:Heun-Darboux} already appeared in Darboux's 1882 paper \cite{a-Darboux1882} which predates Heun's work. In the present paper, however, following Ruijsenaars and Takemura we call the operator \eqref{eq:Heun-Darboux} the Heun Hamiltonian.

In their seminal work in the seventies Olshanetsky and Perelomov established quantum integrability of rational and trigonometric Calogero-Moser systems for special fixed parameters, see their review paper \cite{a-OlshPer83}. They showed that these Hamiltonians can be obtained as radial parts of Casimir elements for Riemannian symmetric spaces $G/K$. This observation constitutes an important half-way mark between the classical work by Gelfand and Harish-Chandra on harmonic analysis on symmetric spaces, e.g.~\cite{a-Gelfand50}, \cite{a-HarishChandra58}, and the more recent development of Heckman-Opdam theory which established quantum integrability of the Calogero-Moser system algebraically for general coupling constants \cite{a-HeckmanOpdam87}, \cite{a-Opdam88}, \cite[Chapter I.2]{b-HeckSchlicht94}. Heckman-Opdam  theory naturally led to the replacement of the geometric approach based on symmetric spaces by an algebraic approach based on Dunkl operators. This in turn led to the introduction of Cherednik algebras and more general double affine Hecke algebras. 

The range of coupling constants for which Calogero-Moser Hamiltonians are obtained by radial part calculations in the spirit of Olshanetsky and Perelomov may be enlarged by considering non-trivial $K$-types in the sense of \cite[Section 5]{b-HeckSchlicht94}. Here we call zonal spherical functions for non-trivial $K$-types twisted zonal spherical functions. For Hermitian symmetric spaces this leads to a discrete family of rational or trigonometric Hamiltonians \cite[Theorem 5.1.7]{b-HeckSchlicht94}. Initially, elliptic Calogero-Moser operators were excluded, but in the introduction of \cite{a-OlshPer83} Olshanetsky and Perelomov suggested that the elliptic case may be treated in a similar way involving Kac-Moody algebras. This suggestion was picked up in an impressive series of papers by Etingof, Frenkel, and Kirillov, e.g.~\cite{a-EK94}, \cite{a-EFK95}, \cite{a-EK95}, \cite{phd-Kirillov95}, where the authors calculate radial parts of Casimir operators for Kac-Moody versions of symmetric spaces of the form $G/ K$ with $G=K\times K$ where $K$ is diagonally embedded. For $\slzh$ they obtained operators of the form
\begin{align}\label{eq:KZB-heat}
  -4 \pi i K \frac{\partial}{\partial \tau} + \frac{\partial^2}{\partial x^2} - m(m+1) \wp(\tau/2 + x)
\end{align}
where now $K$ denotes the level, $\tau$ and $1$ are the periods of the Weierstrass $\wp$-function, and $m\in \N$ parametrizes irreducible $\slfrak_2$-modules, see \cite[p.~586]{a-EK94}. The operator \eqref{eq:KZB-heat} provides the one-dimensional version of the KZB-heat equation which first appeared in a paper by Bernard \cite{a-Bernard88}, see also \cite{a-FelderWiecz96} or \cite[p.~70]{b-Varchenko03}. Cherednik initiated a discussion of the KZB-heat equation using degenerate double affine Hecke algebras \cite{a-Cherednik95b} also making contact with the results by Etingof and Kirillov. In their construction Etingof and Kirillov identified the underlying zonal spherical functions with affine analogues of Jack polynomials which they characterized analogously to Jacobi polynomials in Heckman-Opdam theory, see \cite[I.1.3]{b-HeckSchlicht94}. While Jacobi polynomials provide a basis of Weyl group invariant polynomials, the affine Jack polynomials in \cite{a-EK95} provide a basis of the space of Looijenga's symmetric theta functions \cite{a-Looijenga76}, \cite{a-Looijenga80}. 

In the present paper we consider the affine Lie algebra $\slzh$ together with the Chevalley involution $\thetah$. This is the simplest example of an affine symmetric pair which goes beyond the diagonal case $K\times K/ K$  considered by Etingof and Kirillov. The Lie subalgebra $\slzhth$ of $\slzh$ fixed under $\thetah$ is the Onsager algebra and has a two-parameter family of one-dimensional representations $\{\chi_{a,b}\,|\, a,b,\in \C\}$, see Corollary \ref{cor:OnsRep}. For any pair of one-dimensional representation $\eta=\chi_{a_0,a_1}$, $\chi=\chi_{b_0,b_1}$ of $\slzhth$ we define a set of twisted zonal spherical functions ${}^\eta \overline{\C}[\SLzh]^\chi_K$ of level $K$ using the Peter-Weyl Theorem and a completion procedure, see \eqref{eq:tw-zon-def}. We calculate the radial part of the Casimir element of $\slzh$ acting on ${}^\eta \overline{\C}[\SLzh]^\chi_K$. After conjugation by a suitable element $\deltah$, the resulting operator can be identified with a Heun version of the KZB-heat operator
\begin{align}\label{eq:HeunKZBheat}
  -4 \pi i K \frac{\partial}{\partial \tau} + \frac{\partial^2}{\partial x^2} - \sum_{i=0}^3 l_i(l_i+1) \wp_i(x) +c(\tau)
\end{align}
where $c(\tau)$ is holomorphic in $\tau$ and the coupling constants $l_i$ are related to the characters $\eta,\chi$ via the relations
\begin{align*}
   \begin{aligned}
  l_0&=(a_1-b_1-1)/2, & l_1&=(a_1+b_1-1)/2, \\  l_2&=(a_0+b_0-1)/2, & l_3&=(a_0-b_0-1)/2,
  \end{aligned}
\end{align*}
see Theorems \ref{thm:del-Om} and \ref{thm:HeunKZBheat}. Twisted zonal spherical functions provide eigenfunctions for this operator. We show for $\eta=\chi$ that these eigenfunctions are analytic in a suitable domain. We moreover show how twisted zonal spherical functions provide a basis of subspaces of the space of Looijenga's symmetric theta functions, see Corollary \ref{cor:basisSymTheta}. 

The operator \eqref{eq:HeunKZBheat} has appeared in the literature previously in various contexts. This was brought to the author's attention by H.~Rosengren who kindly provided his preprint \cite{a-Rosengren13p} ahead of publication. We refer the reader to the introduction of \cite{a-Rosengren13p} and to the recent paper \cite{a-LangTak12} for an overview. 

The present paper takes its main inspiration from Etingof's and Kirillov's work \cite{a-EK95} and from Kirillov's more detailed thesis \cite{phd-Kirillov95}. The second main ingredient is the beautiful presentation of Harish-Chandra's radial part calculations given by Casselman and Mili{\v c}i{\'c} in \cite{a-CM82}. Their algebraic treatment includes twisted zonal spherical functions in full generality and translates easily into the setting of Kac-Moody algebras. Compared to \cite{a-EK95} we take a pedestrian approach. As in \cite{b-Varchenko03} we intentionally restrict to $\slzh$ with the Chevalley involution, mainly because this is a crucial test case which avoids many technicalities such as restricted root systems. It may however be expected that most results of the present paper can be generalized to involutive automorphisms of the second kind of symmetrizable Kac-Moody algebras as classified by Kac and Wang in \cite{a-KW92}, see also \cite{a-BBBG95}, \cite[Theorem 2.6 ]{a-Kolb12p}.

In \cite[Section 11]{a-EK95} Etingof and Kirillov give some indication of how to extend their construction to quantum affine algebras. The setting of the present paper leads to the $q$-Onsager algebra \cite{a-Baseilhac05} which is the simplest example of a quantum symmetric pair coideal subalgebra for an affine Kac-Moody algebra \cite[Example 7.6]{a-Kolb12p}. In the finite case radial part computations for quantum symmetric pairs were performed by Noumi \cite{a-Noumi96} and Letzter \cite{a-Letzter04} in order to give interpretations of Macdonald polynomials in terms of quantum groups. The present paper may serve as a guide to extend Letzter's construction to the affine Kac-Moody setting or at least to the $q$-Onsager algebra. In radial part calculations for quantum groups differential operators are replaced by difference operators. Ruijsenaars and Schneider introduced a class of difference operator analogs of elliptic Calogero-Moser systems, commonly referred to as relativistic Calogero-Moser systems \cite{a-RuijSchneider86}, \cite{a-Ruijsenaars87}. A relativistic version of the Inozemtsev model was introduced by van Diejen \cite{a-vanDiejen94}. One may expect that a radial part calculation for the $q$-Onsager algebra reproduces van Diejen's model for $N=1$. Elliptic difference operators also appeared in Cherednik's work \cite{a-Cherednik95} which also hints at relations to $q$-Kac-Moody algebras. 

We now describe the contents and structure of the present paper in more detail. In Section \ref{sec:slzh}, after fixing notations, we study one-dimensional representation of the Onsager algebra inside completed irreducible highest weight representations $\overline{V(\lambda)}$ of $\slzh$ where $\lambda$ denotes a dominant integral weight. Proposition \ref{prop:aff-inv} determines precisely which one-dimensional representations of $\slzhth$ occur inside $\overline{V(\lambda)}$ showing that this is analogous to the situation for finite dimensional Gelfand pairs. We use these results in Section \ref{sec:genChar} to determine twisted zonal spherical functions for $\slzh$. The Peter-Weyl decomposition serves as the definition of the relevant algebra of functions, but again the construction involves completions which are systematically described in Subsection \ref{sec:grad-completion}. Subsequently, we review the construction and properties of symmetric theta functions for $\slzh$ as considered more generally in \cite{a-Looijenga80}, \cite{a-EK95}. With these preparations, twisted zonal spherical functions are define and classified in Subsection \ref{sec:twisted-zonal} and it is shown in Subsection \ref{sec:gen-char} how they give rise to symmetric theta functions via a generalized character map. 

Section \ref{sec:rad} is devoted to the radial part computation for the Casimir element $\Omega$ of $\slzh$. Here we follow \cite{a-CM82}, reformulating their exposition in the Kac-Moody setting for $\slzh$. For one-dimensional representations $\eta=\chi_{a_0,a_1}$, $\chi=\chi_{b_0,b_1}$ the radial part $\Pi_{\eta,\chi}(\Omega)$ is calculated in two steps, first considering the case $\eta=\chi=0$ and then determining additional terms which appear in the general case. Following \cite{a-EK95} and \cite{phd-Kirillov95} we conjugate the resulting operator to eliminate first order differentiation, see Proposition \ref{prop:Om-conj} and Theorem \ref{thm:del-Om}. Up to this point the radial part $\Pi_{\eta,\chi}(\Omega)$ is an operator acting on a completion $\overline{\C}[P]$ of the group algebra of the weight lattice. In Section \ref{sec:functional}, again following \cite{a-EK95}, \cite{phd-Kirillov95}, we interpret $\Pi_{\eta,\chi}(\Omega)$ as the Heun KZB-heat operator \eqref{eq:HeunKZBheat}, see Theorem \ref{thm:HeunKZBheat}. In the case $\chi=\eta$ we show moreover that twisted zonal spherical functions can be interpreted as holomorphic functions on a suitable domain. The restriction to the symmetric case $\eta=\chi$ is necessary to obtain Weyl group invariance. 

The paper ends with an appendix in which we collect well-known presentations of theta functions and the corresponding Weierstrass $\wp$-functions following \cite{b-Mum83}. 

\medskip

\noindent{\bf Acknowledgments.} The author is grateful to J.V.~Stokman for detailed discussions on the subject of this paper and for encouragement. He also thanks H.~Rosengren for providing the preprint \cite{a-Rosengren13p} ahead of publication.
\section{Spherical representations for  $\widehat{\slfrak}_2$}
\subsection{The Chevalley involution for
  $\widehat{\slfrak}_2$} \label{sec:slzh}   
Let $e,f,h$ denote the standard basis of the complex Lie algebra
$\slfrak_2$ with Lie brackets $[e,f]=h$, $[h,e]=2e$,
$[h,f]=-2f$. Recall that the affine Lie algebra $\slzh$ can be
realized as a central extension of a loop algebra 
\begin{align}\label{cong:slzh-loop}
  \widehat{\slfrak}_2\cong \C[t,t^{-1}]\ot \slfrak_2 \oplus \C c
  \oplus \C d.
\end{align}
Here the commutator bracket is given by
\begin{align}
  [t^m\ot x, t^n\ot y]&=t^{m+n}\ot [x,y] + m \delta_{m,-n}(x,y)c, \label{eq:com1}\\
  [d,t^n\ot x]&=n t^n\ot x, \qquad
  c \mbox{ is central } \label{eq:com2}
\end{align}
where $(\cdot, \cdot)$ denotes the symmetric invariant bilinear form
on $\slfrak_2$ with $(e,f)=1$. Setting
\begin{align*}
  e_0&= t\ot f,& f_0&= t^{-1}\ot e,& h_0&= -1\ot h +
  c, \\
  e_1&= 1\ot e,& f_1&= 1\ot f,& h_1&= 1\ot h
\end{align*}
one obtains the standard generators of $\slzh$ realized via the Cartan
matrix 
\begin{align}\label{eq:sl2-cartan}
  A=(a_{kl})_{k,l\in \{0,1\}}=\left(\begin{array}{rr}
    2 & -2 \\ -2 & 2
  \end{array}\right)
\end{align}
as in  \cite[1.2, 1.3]{b-Kac1}. The Chevalley involution
$\thetah:\slzh\rightarrow \slzh$ is given by $\thetah(d)=-d$ and
\begin{align*}
     \thetah(e_j)&=-f_j,&\thetah(f_j)&=-e_j, &\thetah(h_j)&=-h_j
     &\mbox{for $j=0,1$.}
\end{align*}
We are interested in the fixed Lie subalgebra
\begin{align*}
  {\widehat{\mathfrak{sl}}}{}_2^{\thetah}=\{x\in \slzh\,|\, \thetah(x)=x\}.
\end{align*}
Observe that $\slzhth$ is the Lie subalgebra of $\slzh$
generated by the elements $e_0-f_0$ and $e_1-f_1$. It can be given in
terms of generators and relations as follows.
\begin{prop}
  The Lie algebra $\slzhth$ is isomorphic to the Lie algebra
  generated by two elements $B_0$ and $B_1$ and relations
  \begin{align}\label{eq:slzhRels}
    [B_0,[B_0,[B_0,B_1]]]&= 4 [B_0,B_1], &   [B_1,[B_1,[B_1,B_0]]]&= 4 [B_1,B_0].
  \end{align}
\end{prop}
\begin{proof}
  Let $i\in \C$ denote the complex imaginary unit.
  It is checked by direct computation using the Serre relations and
  the Jacobi identity for $\slzh$ that the generators $B_0=i(e_0-f_0)$ and
  $B_1=i(e_1-f_1)$ satisfy the relations \eqref{eq:slzhRels}. On the
  other hand the triangular decomposition of $\slzh$ implies that the Lie subalgebra $\slzhth$ is generated 
  by the elements $B_0$, $B_1$. Moreover, the positive part of $\slzh$ can be defined in
  terms of the Serre relations. This implies that any relations satisfied by the generators $B_0$
  and $B_1$ in $\slzhth$ are generated by those given by \eqref{eq:slzhRels}.
\end{proof}
\begin{rema}\label{rema:Onsager}
  The Lie algebra $\slzhth$ first appeared in 1944 in Onsager's investigation of the Ising model \cite{a-Onsager44}. Indeed, observe that $\thetah$ is given by
  \begin{align*}
    \thetah(t^n\ot x)=t^{-n}\ot \theta(x), \qquad \thetah(c)=-c, \qquad \theta(d)=-d
  \end{align*}
where $\theta$ denotes the Chevalley involution of $\slfrak_2$. Hence $\slzhth$ has a vector space basis consisting of elements
\begin{align*}
  A_n &= 2 i (t^n\ot e - t^{-n} \ot f)  & & \mbox{for all $n\in \Z$},\\
  G_n &= t^n \ot h - t^{-n} \ot h & & \mbox{for all $n\in \N$}.
\end{align*}
By \eqref{eq:com1}, \eqref{eq:com2} the commutators of these elements are calculated as
  \begin{align}
    [A_n,A_m]=4 G_{n-m}, \quad [G_n,G_m]=0, \quad
    [G_n,A_m] = 2 A_{n+m} - 2 A_{m-n} \label{eq:Onsager}
  \end{align}
which are precisely those given in \cite[(60), (61), (61a)]{a-Onsager44}. In the physics literature the relations \eqref{eq:slzhRels} are called the Dolan-Grady relations, while the Lie algebra given by generators $A_n, G_m$ for $n\in \Z, m\in \N$ and relations \eqref{eq:Onsager} is called the Onsager algebra. \end{rema}
In the following we keep the notation $B_0=i(e_0-f_0)$ and
$B_1=i(e_1-f_1)$. The multiplication by the complex imaginary unit
ensures that $B_j$ has eigenvalues $1$ and $-1$ in the vector
representation of the $\slfrak_2$-triple $e_j,f_j,h_j$.
\begin{cor}\label{cor:OnsRep}
  For every $a,b\in \C$ there exists a one-dimensional representation
  $\chi_{a,b}:\slzhth\rightarrow \C$ uniquely defined by
  $\chi_{a,b}(B_0)=a$ and $\chi_{a,b}(B_1)=b$. Any one-dimensional
  representation of $\slzhth$ is of the form $\chi_{a,b}$ for
  some $a,b\in \C$. 
\end{cor}
For later use we fix some more notation. Let
$\hfrak=\Lin_\C\{h_0,h_1,d\}$ be the 
Cartan subalgebra of $\slzh$ where $\Lin_\C$ denotes the linear
span. The simple roots $\alpha_0,\alpha_1\in \hfrak^\ast$ are defined by
$\alpha_k(h_l)=a_{kl}$, $\alpha_k(d)=\delta_{0,k}$ for $k,l\in
\{0,1\}$. They span the root 
lattice $Q=\Z\alpha_0+\Z\alpha_1$. We will write $Q^+=\N_0 \alpha_0 +
\N_0\alpha_1$.  Let  $P=\{\lambda\in
\hfrak^\ast\,|\,\lambda(h_0),\lambda(h_1),\lambda(d)\in \Z\}$ be the
weight lattice and $P^+=\{\lambda\in P\,|\, \lambda(h_k)\in \N_0
\mbox{ for $k=0,1$}\}$ be the set of dominant integral weights. The
fundamental weights $\varpi_0, \varpi_1\in P$ are defined by
$\varpi_k(h_l)=\delta_{k,l}$ and $\varpi_k(d)=0$ for $k,l\in
\{0,1\}$. Define $\delta=\alpha_0+\alpha_1$. Then $\delta(h_i)=0$,
$\delta(d)=1$, and hence $P^+=\N_0\varpi_0+\N_0\varpi_1+\Z\delta$. 
The weight lattice $P$ is partially ordered by $\mu\le \nu$ if and only 
if there exists  $\alpha\in Q^+$ such that $\mu+\alpha=\nu$.
We write $\Phi$ to denote the root system for $\slzh$, and $\Phi^+$,
$\Phi^-$ to denote the sets of positive and negative roots,
respectively. Recall that $\Phi^+=\{\alpha_1\}\cup\bigcup_{n\in\N} 
(n\delta+ \Phi_0)$ where $\Phi_0=\{-\alpha_1, 0,\alpha_1\}$. 
There are simple reflections $s_{\alpha_0}, s_{\alpha_1}$ acting on
  $\hfrak^\ast$ by
\begin{align*}
  s_{\alpha_j}(\lambda)=\lambda-\lambda(h_j)\alpha_j,\qquad \lambda\in
  \hfrak^\ast, j\in \{0,1\}
\end{align*}
which generate the Weyl group $W$. Setting $r=s_{\alpha_1}$ and
$t_k=(s_{\alpha_1}s_{\alpha_0})^k$ one has an identification $(\Z/2\Z)\ltimes
\Z\rightarrow W$ given by $(l,k)\mapsto r^l t_k$. The symmetric bilinear form on
$\hfrak=\Lin_\C\{h_1, c, d\}$ given by
\begin{align*}
    (h_1,h_1)&=2,& (c,d)&=1, & \mbox{all other pairs vanish}
\end{align*}
is used to identify $\hfrak$ with $\hfrak^\ast$. Under this
identification $(\cdot,\cdot)$ is $W$-invariant and one has
\begin{align}\label{eq:h=hast}
      \varpi_0&= d, & \varpi_1&= d+ h_1/2, & \delta &= c, & \alpha_0&=h_0,& \alpha_1=h_1.
\end{align}
The action of $W$ on $\hfrak$ is given by $s_{\alpha_j}(h)=h-\alpha_j(h)h_j$ and hence
\begin{align}
    r(h_1)&=-h_1,& r(c)&=c,& r(d)&=d,\nonumber\\
    t_k(h_1)&=h_1 + 2kc,& t_k(c)&=c,& t_k(d)&=d - k h_1-k^2c.\label{eq:tk-action}
\end{align}
For later use also observe that $\mu=(K-n)\varpi_0 + n
\varpi_1 + b \delta $ satisfies
\begin{align}\label{eq:tkmu}
   t_k(\mu)= K d +(\frac{n}{2} - kK)h_1 +(-Kk^2  + kn+b)c
\end{align}
where we identify $\hfrak$ with $\hfrak^\ast$ as in
\eqref{eq:h=hast}. Hence, for $\rho=\varpi_0+\varpi_1$ one gets
\begin{align}
        (t_k\mu,\rho)= \frac{n}{2} - kK +2(-Kk^2 + kn+b).\label{eq:rho-eval}
\end{align} 
\subsection{Completions and $\chi$-invariants}
  For any left $\slzh$-module $V$ and any one-dimensional
  representation $\chi:\slzhth\rightarrow \C$ define
  \begin{align*}
    V^\chi:=\{v\in V\,|\,xv=\chi(x)v\quad\mbox{for all $x\in \slzhth$}\}.
  \end{align*}
  We call $V^\chi$ the space of $\chi$-invariants in $V$.
  For any $\lambda\in P^+$ let $V(\lambda)$ denote the integrable
  $\slzh$-module of highest weight $\lambda$. The space of
  $\chi$-invariants in $V(\lambda)$ is trivial if $\lambda\notin \Z\delta$. Indeed, 
  any element of $V(\lambda)$ is a sum of weight vectors. For any non-zero
  $\chi$-invariant, however, at least one weight summand has to be a
  lowest weight vector. This is impossible in $V(\lambda)$. 
  
  To obtain non-trivial $\chi$-invariants we consider a completion of
  $V(\lambda)$ as follows. For any left $\slzh$-module $V$ in the
  highest weight category $\cO$ let $\delta V$ denote its graded
  (category $\cO$) dual. This standard notation should cause no
  confusion with the element $\delta\in \hfrak^\ast$ defined in the
  previous subsection. Consider $\delta V$ as a right $\slzh$-module
  in a lowest weight category. Now define for any left category
  $\cO$-module $V$ a completion 
  \begin{align}\label{eq:vbar}
    \overline{V}:=(\delta V)^\ast
  \end{align}
  where $W^\ast=\Hom_\C(W,\C)$ denotes the linear dual space of a
  vector space $W$. Alternatively, if $V=\oplus_{\mu_\in P}V_\mu$ is
  the weight space decomposition of $V$ then 
  \begin{align*}
    \overline{V}=\prod_{\mu\in P}V_\mu.
  \end{align*}
  It is immediate from \eqref{eq:vbar} that the left $\slzh$-module
  structure on $V$ extends to a left $\slzh$-module structure on
  $\overline{V}$. For the right lowest weight $\slzh$-module $\delta V$
  the completion is defined analogously by
  $\overline{\delta V}=V^\ast$. The following proposition summarizes 
  the structure of $\chi$-invariants in $\overline{V(\lambda)}$ for
  $\lambda\in P^+$. 
\begin{prop}\label{prop:aff-inv}
  Let $\lambda\in P^+$ and $a_0,a_1\in \C$ and set $\chi=\chi_{a_0,a_1}$. 

  1) If $a_0\notin \Z$ or $a_1\notin \Z$ then
  $\overline{V(\lambda)}^\chi=\{0\}$.

  2) If $a_0,a_1\in \Z$ then 
  \begin{align*}
     \dim\big( \overline{V(\lambda)}^\chi\big)=
       \begin{cases}
         1& \mbox{if $\lambda(h_i)-|a_i| \in 2\N_0$ for $i=0,1$},\\ 
         0& \mbox{else.}
       \end{cases}
  \end{align*}
\end{prop}
  The rest of this subsection is devoted to the proof of the above
  proposition. We follow the line of argument given in a related
  finite type setting in \cite[Proof of  Theorem 4.3]{a-Letzter00}. For
  convenience, we break the proof up into  
  several smaller steps. For the rest of this subsection fix
  $a_0,a_1\in \C$, let $\chi=\chi_{a_0,a_1}$ denote the corresponding
  one-dimensional representation of $\slzhth$, and fix
  $\lambda\in P^+$.  Moreover, let $v_\lambda$ be a highest weight vector in $V(\lambda)$. The usual highest weight argument yields the
  following result.  
\begin{lem}\label{lem:le1}
  Let $w$ be a nonzero element in $\overline{V(\lambda)}^\chi$. Then
  there exists $\beta\in \C\setminus\{0\}$ such that $w-\beta v_\lambda\in
  \prod_{\mu<\lambda}V(\lambda)_\mu$. In particular, one has
  $\dim\big( \overline{V(\lambda)}^\chi\big)\le 1$. 
\end{lem}
  To simplify notation define $\cB=U(\slzhth)$ and note that
  $\chi$ yields a one-dimensional representation of $\cB$ which we denote by the same
  symbol. For $j=0,1$ define $\tB_j=B_j-a_j\in\cB$. 
  \begin{lem}\label{lem:pi}
    For $j=0,1$ there exists a monic polynomial $p_j$ of degree
    $\lambda(h_j)+1$ such that $p_j(\tB_j)v_\lambda=0$. The polynomial
    $p_j$ has constant coefficient $p_j(0)=0$ if and only if $a_j\in
    \Z$ and $\lambda(h_j)-|a_j|\in 2\N_0$.
  \end{lem}
\begin{proof}
  For $j=0,1$ let $\gfrak_j\subset \slzh$ denote the Lie subalgebra
  spanned by the $\slfrak_2$-triple $e_j,f_j,h_j$. The subspace
  $U(\gfrak_j)v_\lambda\subseteq V(\lambda)$ is an irreducible $U(\gfrak_j)$-module of dimension
  $\lambda(h_i)+1$. Moreover, the element $B_j$ is the image of
  $h_j$ under an automorphism of $\gfrak_j$. Hence the action of $B_j$ on
  $U(\gfrak_j)v_\lambda$ is diagonalizable with eigenvalues
  $\{\lambda(h_j)-2k\,|\,k=0,1,\dots,\lambda(h_j)\}$. Therefore the
  polynomial 
  \begin{align}\label{eq:pi}
    p_j(t)=\prod_{k=0}^{\lambda(h_j)}(t-\lambda(h_j)+2k+a_j)
  \end{align}
  satisfies $p_j(\tB_j)v_\lambda=0$. The factor $(t-0)$ occurs in the
  product \eqref{eq:pi} if and only if $\lambda(h_j)-|a_j|\in 2\N_0$. 
\end{proof}
For any ordered $k$-tuple $I=(i_1,\dots,i_k)\in \{0,1\}^k$ we write
$|I|=k$ and $f_I=f_{i_1}\dots f_{i_k}$ and $\tB_I=\tB_{i_1}\dots
\tB_{i_k}$. Let $\cJ\subset \bigcup_{k=1}^\infty \{0,1\}^k$ be a
subset such that $\{1\}\cup\{f_J\,|\,J\in \cJ\}$ is a basis of $U(\slzhm)$. In
this case $\{1\}\cup\{\tB_J\,|\,J\in \cJ\}$ is a basis of
$U(\slzhth)$. Observe that for any ordered $k$-tuple $K\in \{0,1\}^k$ one
has by construction
\begin{align}\label{eq:kurzgut}
  \tB_K\in \sum_{J\in \cJ, |J|\le k}\tB_J.
\end{align}
For any $J=(j_1,\dots,j_k)\in \cJ$ and $j=0,1$ define 
\begin{align*}
  \rt_j(J)=\max\{l\in \N_0\,|\,j=j_k=j_{k-1}=\dots=j_{k-l+1}\}.
\end{align*}
Recall that the annihilator in $U(\slzhm)$ of the highest weight
vector $v_\lambda\in V(\lambda)$ has the form
\begin{align*}
  \mathrm{Ann}_{U(\slzhm)}(v_\lambda)&=\sum_{j=0,1}U(\slzhm)f_j^{\lambda(h_j)+1}.
\end{align*}
For $j\in\{0,1\}$ and $k\in \N$ define $M(k,j)=\{K\in \{0,1\}^k\,|\, \rt_j(K)\ge
\lambda(h_j)+1\}$. Then  the above expression can be rewritten as
\begin{align*}
  \mathrm{Ann}_{U(\slzhm)}(v_\lambda)&=\sum_{j=0,1}\sum_{k=1}^\infty
                                  \sum_{K\in M(k,j)}\C f_K.
\end{align*}
Hence we may assume that $\cJ$ contains a subset $\cI'\subset
\bigcup_{j\in\{0,1\}, k\in \N}M(k,j)$ such that $\{f_K\,|\,K\in \cI'\}$ is a
basis of $\mathrm{Ann}_{U(\slzhm)}(v_\lambda)$. In this case
\begin{align*}
  \cI:=\cJ\setminus \cI'=\{J\in \cJ\,|\,\rt_j(J)\le \lambda(h_j)
\mbox{ for $j=0,1$}\}
\end{align*}
is a subset such that $\{v_\lambda\}\cup\{f_Iv_\lambda\,|\,I\in \cI\}$
is a basis of $V(\lambda)$. One shows by induction, that in this case 
$\{v_\lambda\}\cup\{\tB_Iv_\lambda\,|\,I\in \cI\}$ is also basis of $V(\lambda)$. 
\begin{lem}\label{lem:notin-ker}
  One has $v_\lambda\notin \ker(\chi)v_\lambda$ if an only if $a_j\in
  \Z$ and $\lambda(h_j)-|a_j|\in 2\N_0$ for $j\in\{0,1\}$.
\end{lem}
\begin{proof}
  Assume that $\lambda(h_j)-|a_j|\notin 2\N_0$ or $a_j\notin \Z$ for one
  $j\in\{0,1\}$. By Lemma \ref{lem:pi} there exists a monic polynomial $p_j$
  with nonvanishing constant coefficient such that
  $p_j(\tB_j)v_\lambda=0$. Together with $\tB_j\in \ker(\chi)$ this
  implies $v_\lambda\in \ker(\chi)v_\lambda$.

  Conversely, assume that $a_j\in \Z$ and $\lambda(h_j)-|a_j|\in
  2\N_0$ for both $j=0$ and $j=1$. In this case $p_j(0)=0$ for $j=0,1$
  by Lemma \ref{lem:pi}. As $\{1\}\cup \{\tB_J\,|\,J\in \cJ\}$ is a
  basis of $\cB$ one gets 
  \begin{align*}
    \ker(\chi)=\bigoplus_{J\in \cJ} \C \tB_J=\sum_{I\in \cI} \C
    \tB_I+\sum_{J\in \cI'}\C \tB_{J}.
  \end{align*}
  For any $J\in \cI'$ we get 
  \begin{align*}
    \tB_J\in& \sum_{j\in\{0,1\}}\cB p_j(\tB_j) + \sum_{1\le |K|<|J|}\C
    \tB_K\\
    &= \sum_{j\in\{0,1\}}\cB p_j(\tB_j) + \sum_{J'\in\cJ,  |J'|<|J|}\C\tB_{J'}
  \end{align*}
  where we used the fact that $p_j(0)=0$ to obtain the first inclusion
  and Relation \eqref{eq:kurzgut} for the equality of sets. Induction
  over $|J|$ yields
  \begin{align*}
    \ker(\chi)=\sum_{I\in \cI}\C\tB_I +\sum_{j\in\{0,1\}}\cB p_j(\tB_j).
  \end{align*}
  Hence the relation $p_j(\tB_j)v_\lambda=0$ implies that for any $b\in
  \ker(\chi)$ one has
  \begin{align*}
    bv_\lambda \in \sum_{I\in \cI} \C\tB_I v_\lambda.
  \end{align*}
  This expression is linearly independent from $v_\lambda$ as observed
  just above the statement of the Lemma. Hence $v_\lambda$ can not be
  written as $bv_\lambda$ for any $b\in \ker(\chi)$. 
\end{proof}
After these preliminary considerations we are in a position to prove
the main result of this subsection.

\medskip

\noindent\textit{Proof of Proposition \ref{prop:aff-inv}.}   Assume
that $a_j\notin \Z$ or $\lambda(h_j)-|a_j|\notin 2\N_0$ for one
$j\in\{0,1\}$. In this case $U(\gfrak_j)v_\lambda$ does not contain
any non-zero $\tB_j$-invariant element. By Lemma \ref{lem:le1} this
implies that $\overline{V(\lambda)}^\chi=\{0\}$.

Conversely, assume that $a_j\in \Z$ and $\lambda(h_j)-|a_j|\in 2\N_0$
for both $j=0$ and $j=1$. Let $\ochi:\cB\rightarrow \C$ be the
one-dimensional representation defined by 
$\ochi(B_j)=-a_j$. By Lemma \ref{lem:notin-ker} one has
$v_\lambda\notin \ker(\ochi)v_\lambda$. Let $S:V(\lambda)\times
V(\lambda)\rightarrow \C$ denote the nondegenerate contravariant
bilinear (Shapovalev) form for $V(\lambda)$, which is denoted by $B$
in \cite[Section 9.4]{b-Kac1}. Recall that contravariance means that
\begin{align}\label{eq:contravariance}
  S(xv,w)=-S(w,\thetah(x)v)\qquad \mbox{for all}\quad v,w\in
  V(\lambda), x\in \slzh.
\end{align}
The canonical projection $\phi:V(\lambda)\rightarrow
V(\lambda)/(\ker(\ochi)v_\lambda)$ defines a nonzero linear functional on
$V(\lambda)$. As $S$ is nondegenerate and orthogonal with respect to
the weight space decomposition \cite[Proposition 9.4]{b-Kac1} there
exists a uniquely determined element $v\in
\overline{V(\lambda)}\setminus\{0\}$ such 
that $S(v,w)=\phi(w)$ for all $w\in V(\lambda)$. We claim that $v\in 
\overline{V(\lambda)}^\chi$. It suffices to show that
$S((B_j-a_j)v,w)=0$ for all $w\in V(\lambda)$. To this end one writes
$w=a_\lambda v_\lambda + \sum_{I\in \cI} a_I \tB_I v_\lambda$ for some
coefficients $a_\lambda,a_I\in \C$ and calculates using
\eqref{eq:contravariance}
\begin{align*}
  S((B_j-a_j)v,w)&=-S(v,(\theta(B_j)+a_j)w)=-S(v,(B_j+a_j)w)=-\phi((B_j+a_j)w)\\
                 &=-\phi((B_j+a_j)a_\lambda v_\lambda)-\phi((B_j+a_j)\sum_{I\in\cI}
                 a_I\tB_I v_\lambda)=0
\end{align*}
as $\phi$ vanishes on $\ker(\ochi)v_\lambda$. The fact that $v\in
\overline{V(\lambda)}^\chi\setminus\{0\}$ together with Lemma
\ref{lem:le1} complete the proof of Proposition \ref{prop:aff-inv}. \hfill $\square$
\section{Generalized characters and symmetric theta functions}\label{sec:genChar}
Our next aim is to construct twisted zonal spherical functions for $\slzht$, and to show how they give rise to classes of formal symmetric theta functions. In the present affine situation such constructions involve completions. We discuss these completions in Subsection \ref{sec:grad-completion}. The definition of formal symmetric theta functions is given in Subsection \ref{sec:sym-theta}. All the material of Subsection \ref{sec:sym-theta} is contained in \cite{a-EK95}, \cite{phd-Kirillov95}, and \cite{a-Looijenga80} in more generality. Here, for simplicity, we continue to restrict to the case of $\slzh$, and we include complete proofs to make the present text self-contained. In Subsection \ref{sec:frac-completion} we introduce the completion of a localization of $\C[P]$ which will be needed to discuss the radial part of the Casimir element in Section \ref{sec:rad}. Twisted zonal spherical functions are defined in Subsection \ref{sec:twisted-zonal} and their relation to formal symmetric theta functions is established in Subsection \ref{sec:gen-char}.

In analogy to the finite case define
\begin{align}\label{eq:CSLzh}
    \C[\SLzh]=\bigoplus_{\lambda\in P^+}\delta V(\lambda) \ot V(\lambda).
 \end{align}  
  The map $\delta V(\lambda)\ot V(\lambda)\mapsto U(\slzh)^\ast$,
  $f\ot v\mapsto c_{f,v}:=(u\mapsto f(uv))$ gives an embedding
  $\iota:\C[\SLzh]\hookrightarrow U(\slzh)^\ast$, since $\delta
  V(\lambda) \ot  V(\lambda)$ are irreducible and inequivalent 
  $U(\slzh)$-bimodules. The image of $\iota$ is an algebra with
  $c_{f,v}c_{g,w}=c_{f\ot g,v\ot w}$. We define 
  an algebra structure on $\C[\SLzh]$ by demanding that $\iota$
  is an algebra homomorphism.
  \begin{rema}
    Let $\SLzh$ be the Kac-Moody group associated to $\slzh$ as in
    \cite[1C]{a-KP83}. In \cite[Theorem 1]{a-KP83} the algebra
    $\C[\SLzh]$ was identified with the so called strongly regular
    functions on $\SLzh$, see also \cite[1.7]{a-Mokler03}.
  \end{rema}
\subsection{Gradings and completions}\label{sec:grad-completion}
To consider affine versions of zonal spherical functions one needs a completion of $\C[\SLzh]$ which contains
$\overline{\delta V(\lambda)}\ot \overline{V(\lambda)}$ for any $\lambda\in P^+$. Restriction to $\hfrak$ will further require a completion of the group algebra of the weight lattice $\C[P]$. In this subsection we give a uniform construction of such completions. 

For $K\in \Z$ set $P_K=\{\lambda\in P\,|\,\lambda(c)=K\}$ and $P_K^+=P^+\cap P_K$. Assume that $V=\oplus_{\gamma \in P} V_\gamma$ is a $P$-graded vector space and set $V_K=\oplus_{\gamma\in P_K}V_\gamma$. For $v=\sum_{\gamma\in P_K}v_\gamma\neq 0$ with $v_\gamma\in V_\gamma$ define
\begin{align*}
  n_1(v)=\max\{(\gamma,\rho)\,|\,v_\gamma\neq 0\}
\end{align*}
where as before $\rho=\varpi_0+\varpi_1$, and for $v, w\in V_K$ set
\begin{align*}
  d_1(v,w)=\begin{cases}
             0 & \mbox{if $v=w$,}\\
             2^{n_1(v-w)}& \mbox{else.}
           \end{cases}
\end{align*}
One checks that $(V_K,d_1)$ is a metric space. Now define $\oV_K$ to be the completion of $V_K$ with respect to the metric $d_1$ and set
\begin{align*}
 \oV=\bigoplus_{K\in \Z} \oV_K.
\end{align*}
\begin{rema}
  Elements of $\oV_K$ can be thought of as infinite series $\sum_{\gamma\in P_K}v_\gamma$ with $v_\gamma\in V_\gamma$ such that for any $N\in \Z$ there are only finitely many $\gamma\in P_K$ with $v_\gamma \neq 0$ and $(\gamma,\rho)>N$. Observe that generally $\oV$ is smaller than the completion of $V$ with respect to $d_1$.
\end{rema}
If $V$ is a $P$-graded algebra then the multiplication $V_K\times V_L\rightarrow V_{K+L}$ is continuous. Hence the completion $\oV$ retains the algebra structure.
\begin{eg}
  If $V=V(\lambda)$ is an integrable highest weight representation corresponding to some $\lambda\in P^+$ then $\oV=(\delta V)^*$ coincides with the completion defined by \eqref{eq:vbar}.
\end{eg}
\begin{eg}\label{eg:oCP}
  Let $V=\C[P]$ be the group algebra of the weight lattice with its natural $P$-grading. Let  $\{e^\lambda\,|\,\lambda\in P\}$ be the standard basis of $\C[P]$ with multiplication $e^\lambda e^\mu=e^{\lambda+\mu}$.  In this case $\oV=\oplus_{K\in \Z}\overline{\C}[P_K]$ with
  \begin{align*}
    \overline{\C}[P_K]=\big\{\sum_{\lambda\in P_K}a_\lambda
  e^\lambda\,|\,\forall N\in\Z\quad \#\{\lambda\in
  P_K\,|\,a_\lambda\neq 0, (\lambda,\rho)\ge N\}<\infty\big\}
  \end{align*}
where $a_\lambda\in \C$ and infinite sums are allowed. In this case we will write $\overline{\C}[P]=\oV$. It coincides with the completion defined in \cite[Section 5]{a-EK95}, \cite[6.2]{phd-Kirillov95}.
\end{eg}
\begin{eg}
  Consider $V=\C[\SLzh]$ as defined \eqref{eq:CSLzh} with graded components
    \begin{align*}
      V_\gamma=\bigoplus_{\lambda\in P^+}\bigoplus_{
      \makebox[1cm]{\small $\alpha,\beta\in P, \atop \alpha+\beta=\gamma$}}
       \delta V(\lambda)_\alpha \ot V(\lambda)_\beta
    \end{align*}
  where $\delta V(\lambda)_\alpha\subseteq \delta V(\lambda)$ denotes the dual space of the weight space $V(\lambda)_\alpha$. In this case $V_K=0$ if $K\notin 2\N_0$ and
  \begin{align*}
      V_{K}=\bigoplus_{\lambda\in P^+_{K/2}} \delta V(\lambda) \ot V(\lambda) \qquad \mbox{if $K\in 2\N_0$.}
  \end{align*}  
  We write $\overline{\C}[\SLzh]_K=\oV_{2K}$ and $\overline{\C}[\SLzh]=\oV$. Elements of $\overline{\C}[\SLzh]_K$ can be
thought of as infinite series
\begin{align*}
  \sum_{\alpha,\beta\in P}\sum_j f_{\alpha,j}\ot v_{\beta,j}
\end{align*}
where $f_{\alpha,j}\in \delta V(\lambda_j)_\alpha$, $v_{\beta,j}\in
V(\lambda_j)_\beta$ for some $\lambda_j\in P^+_K$ and where for any
$n\in \N$ all but finitely many non-vanishing summands $f_{\alpha,j}\ot v_{\beta,j}$
satisfy $(\alpha+\beta,\rho)<-n$.
\end{eg}
\begin{rema}
  Note that $\bigoplus_{\lambda\in P^+}\overline{\delta(V(\lambda))} \ot
  \overline{V(\lambda)}$ is contained in
  $\overline{\C}[\SLzh]$. However, this subspace is not closed under multiplication and $\bigoplus_{\lambda\in P^+_K}\overline{\delta(V(\lambda))} \ot \overline{V(\lambda)}$ is not complete under
  the metric of $\overline{\C}[\SLzh]_K$.    
\end{rema}
\begin{eg}\label{eg:hayashi}
  For $\zeta\in Q$ consider the weight space $V=U(\slzh)_\zeta$. The triangular decomposition $U(\slzh)=U(\nfrak^-)\ot U(\hfrak)\ot U(\nfrak^+)$ allows us to define a $P$-grading of $V$ by 
  \begin{align*}
    V_\mu= U(\nfrak^-)_{\zeta+\mu}\ot U(\hfrak)\ot U(\nfrak^+)_{-\mu}.
  \end{align*} 
  Observe that $V_K=\{0\}$ if $K\neq 0$ and $V_0=V$. 
  Define $\overline{U}(\slzh)_\zeta=\oV$ and set $\overline{U}(\slzh)=\bigoplus_{\zeta\in Q} \overline{U}(\slzh)_\zeta$. As  the multiplication $U(\slzh)_\zeta\times U(\slzh)_{\zeta'} \rightarrow U(\slzh)_{\zeta+\zeta'}$ is continuous, one obtains an algebra structure on $\overline{U}(\slzh)$. This algebra coincides with the completion denoted by $\widehat{U}(\slzh)$ in \cite[1.4]{a-Hayashi88}. 
\end{eg}
\subsection{Symmetric theta functions}\label{sec:sym-theta}
The Weyl group $W$ acts on $\C[P]$, however this action does not survive under the completion described in Example \ref{eg:oCP}. In this subsection we define a subalgebra $A$ of $\overline{\C}[P]$ which retains the $W$-action. In Subsection \ref{sec:gen-char} we will identify zonal spherical functions with elements of $A$. The $W$-invariant elements in $A$ will be called symmetric theta functions.  We follow Etingof's and Kirillov's presentation \cite{a-EK95}, \cite{phd-Kirillov95} of the construction of symmetric theta functions originally given by Looijenga in \cite{a-Looijenga76}, \cite{a-Looijenga80}.
  
As a vector space the completion $\overline{\C}[P]$ defined in Example \ref{eg:oCP} may be considered as a subspace of $\widetilde{\C[P]}:=\prod_{\mu\in P}\C e^\mu$. The Weyl group $W$ acts
on $\widetilde{\C[P]}$ and we consider  $w(\overline{\C}[P])$ as
a subspace of $\widetilde{\C[P]}$ for any $w\in W$. Hence, one may define
\begin{align*}
  A=\bigcap_{w\in W}w(\overline{\C}[P]), \quad  A_K=\bigcap_{w\in
    W}w(\overline{\C}[P_K]). 
\end{align*}
By construction, $A$ is a $\Z$-graded algebra with homogeneous components $A_K$
for $K\in \Z$. Moreover the Weyl group $W$ acts on $A$ and on
$A_K$. Let $A^W\subset A$ and $A^W_K\subset A_K$ denote the subsets of
$W$-invariant elements in $A$ or $A_K$, respectively. Following \cite{a-EK95}, \cite{phd-Kirillov95} we call
elements of $A^W_K$ symmetric theta functions of level $K$.

For any module $V$ in the highest weight category $\cO$ define
\begin{align*}
  \ch(V)=\sum_{\lambda\in P} \dim(V_\lambda)e^\lambda.
\end{align*}
By definition of $\cO$ one has $\ch(V)\in \overline{\C}[P]$.
\begin{lem} [{\cite[Section 5, Example 2]{a-EK95}}]\label{lem:inA}
  Let $V\in Ob(\cO)$. Then $\ch(V)\in A$ if and only if $V$ is
  integrable. In this case $\ch(V)\in A^W$. 
\end{lem}
\begin{proof}
  If $V$ is integrable then $\ch(V)$ is $W$-invariant by
  \cite[Proposition 10.1]{b-Kac1} and hence $\ch(V)\in A^W$. If $V$ is
  not integrable then there exists a simple root $\alpha_j$ and $\lambda\in P$
  such that $\dim(V_{\lambda-n\alpha_j})\neq 0$ for all $n\in
  \N$. Fix $N\in \N$. Then $s_{\alpha_j}\ch(V)$ has infinitely many summands
  $a_{\lambda-n\alpha_j} e^{s_{\alpha_j}\lambda+n\alpha_j}$ for which
  $a_{\lambda-n\alpha_j}\neq 0$ and
  $((s_{\alpha_j}\lambda+n\alpha_j),\rho)=(s_{\alpha_j}\lambda,\rho)+n\ge
  N$. Hence $s_{\alpha_j}\ch(V)\notin \overline{\C}[P]$ which implies
  $\ch(V)\notin A$.
\end{proof}
\begin{rema}\label{rem:inA}
   Let $V\in Ob(\cO)$ be an integrable highest weight module. Consider
   an element $p$ in $\overline{\C}[P]$ of the form $p=\sum_{\mu\in
     P}a_\mu\dim(V_\mu) e^\mu$ for some coefficients $a_\mu\in \C$. 
   By the above lemma $\sum_{\mu\in P}\dim(V_{w\mu}) e^\mu=\ch(V)\in \overline{\C}[P]$ and hence
   $w^{-1}p=\sum_{\mu\in P}a_{w\mu}\dim(V_{w\mu}) e^\mu\in \overline{\C}[P]$ for all $w\in W$. Hence $p\in A$.
\end{rema}
\begin{lem}[{\cite[Lemma 5.1]{a-EK95}}]\label{lem:AW0}
If $K<0$ then $A^W_K=\{0\}$. Moreover, one has
  \begin{align*}
    A_0^W=\Big\{\sum_{n\le n_0}a_n e^{n\delta}\,\Big|\,a_n\in \C, n_0\in
      \N \Big\}
  \end{align*}
  which is a field.
\end{lem}
\begin{proof}
  Assume $\sum_{\lambda\in P_K}a_\lambda e^\lambda\in A_K^W$ and
  $\mu=(K-n)\varpi_0 + n 
    \varpi_1 + b \delta\in P_K$. Then Equation \eqref{eq:rho-eval}
    implies for $K<0$ that
  \begin{align*}
    (t_k\mu,\rho)= \frac{n}{2} - kK +2(-Kk^2 + kn+b)
    \longrightarrow \infty\qquad \mbox{for $k\rightarrow \infty$}.
  \end{align*}  
  By definition of $\overline{\C}[P_K]$ this forces
  $a_\mu=0$. Similarly, for $K=0$ and $n\neq 0$, one has
  \begin{align*}
      (t_k\mu,\rho)= \frac{n}{2} + 2(kn+ b)
    \longrightarrow \infty\qquad \mbox{for $\mathrm{sgn}(n)k\rightarrow \infty$}.
  \end{align*}
  where $\mathrm{sgn}(n)$ denotes the sign of $n$. Again this forces
  $a_\mu=0$. On the other hand $w\delta=\delta$ for all $w\in W$ and
  $(\delta,\rho)=2$ which implies that for any $n_0\in \N$ and any $a_n\in
  \C$ the formal series $\sum_{n\le n_0}a_n e^{n\delta}$ does indeed
  belong to $A_0^W$. These series form a field.
\end{proof}
Let $K\in \N_0$. For $\lambda\in P^+_K$ define the orbit sum
\begin{align*}
  m_\lambda=\sum_{\mu\in W\lambda}e^\mu.
\end{align*}
By Remark \ref{rem:inA} one has $m_\lambda\in A^W_K$. We want to show that suitable elements $m_\lambda$ form a basis of $A^W_K$ over the field $A_0^W$. To this end we will need the following well known result.
\begin{lem}[{\cite[Lemma 4.2]{a-EK95}}]  \label{lem:W-orbit} 
    (1) The Weyl group $W$ leaves $P_K$ invariant.\\
    (2) Let $K\in \N$. Then every $W$-orbit in $P_K$ contains a unique
    element in $P^+_K$.
\end{lem}
\begin{proof}
  Property (1) holds because $P$ is invariant under the Weyl group
  action and the action of $s_{\alpha_j}$ just changes $\lambda\in P$ by
  a  multiple of $\alpha_j$. For Property (2) note that $\lambda\in
  P_K$ implies that $(\lambda,\delta)=K>0$. Hence,
  $(\lambda,\alpha)<0$ only for finitely many positive roots
  $\alpha\in \Phi^+$. Now the claim follows from \cite[Prop. 3.12
  (b),(c)]{b-Kac1}.   
\end{proof}
Define a subset of $P^+_K$ as follows
\begin{align}
  \oP_K^+&=\{\lambda\in P^+_K\,|\,\lambda(d)=0\}\label{eq:oP+K-def1}\\
         &=\{\lambda\in P^+\,|\,\lambda = (K-n_1)\varpi_0 + n_1 \varpi_1\,\mbox{for some $n_1\in \N$}\}.\nonumber
\end{align}
\begin{prop}[{\cite[Theorem 5.2]{a-EK95}}]\label{prop:AWKbase}
  Let $K\in \N_0$. The set $\{m_\lambda\,|\,\lambda\in \oP^+_K\}$ is a
  basis of $A^W_K$ over the field $A_0^W$.
\end{prop}
\begin{proof}
  By Lemma \ref{lem:W-orbit} the elements $\{m_\lambda\,|\,\lambda\in
  \oP^+_K\}$ are linearly independent over $A_0^W$. Let $m=\sum a_\mu e^\mu\in
  A^W_K\setminus\{0\}$. By Lemma \ref{lem:W-orbit} there exists
  $\lambda'\in P^+_K$ such that $a_{\lambda'}\neq 0$. Choose
  $\lambda\in \oP^+_K$ such that $\lambda'\in
  \lambda+\Z\delta$. The set $\{n\in 
  \Z\,|\,a_{\lambda+n\delta}\neq 0\}$ is bounded from above because
  $(\delta,\rho)=2$. Hence, in view of Lemma \ref{lem:AW0}, there
  exists $f_{\lambda}\in A^W_0$ such that $\sum_\mu a_\mu'
  e^\mu:=m-f_\lambda m_{\lambda}$ satisfies $a_\mu'=0$ for all $\mu\in
  \lambda+\Z \delta$. Repeating this procedure one obtains elements
  $f_{\lambda}\in A^W_0$ for all $\lambda\in \oP^+_K$ such that
  $\sum a''_\mu e^\mu:= m-\sum_{\lambda\in \oP^+_K} f_{\lambda}m_{\lambda}$ satisfies
  $a''_\mu=0$ for all $\mu\in \oP^+_K+\Z\delta$. As $\oP^+_K+\Z \delta
  = P^+_K$ Lemma \ref{lem:W-orbit} now implies that $m-\sum_{\lambda\in
    \oP^+_K} f_\lambda m_\lambda=0$.  
\end{proof}
  An element $x\in \overline{\C}[P]$ is called $W$-antiinvariant if
  $wx=\mathrm{sgn}(w)x$ for all $w\in W$. Let $A^{-W}$ denote the set of $W$-antiinvariant element of $\overline{\C}[P]$ and observe that $A^{-W}\subset A$. For example, the elements 
  \begin{align}\label{eq:deltah12}
    \deltah_1=e^{\rho}\prod_{\alpha\in
    \Phi^+}(1-e^{-\alpha})\qquad \mbox{and} \qquad \deltah_2=e^{2\rho}\prod_{\alpha\in
    \Phi^+}(1-e^{-2\alpha})
  \end{align}   
  belong to $A^{-W}$. Define $A^{-W}_K=A^{-W}\cap A_K$  and observe that $A^{-W}_0=\{0\}$ and that $A^{-W}=\bigoplus_{K\in \N} A^{-W}_K$ is a direct sum of $A^W_0$-vector spaces. The following result can be found in \cite[6.2, Example 4]{phd-Kirillov95}, \cite[Theorem 4.2 (ii)]{a-Looijenga80}.
\begin{prop}
  $A^{-W}$ is a free $A^W$-module generated by the element $\deltah_1$.
\end{prop}
\begin{proof}
  If $K\in \N$ and $\lambda\in \overline{P}^+_K$ then the stabilizer
  $\mathrm{Stab}_W(\lambda)$ is finite, see \eqref{eq:tkmu}. Hence we may define
  \begin{align*}
    m^-_\lambda = \sum_{w\in W} \mathrm{sgn}(w)e^{w\lambda}.
  \end{align*}
  Analogously to the proof of Proposition \ref{prop:AWKbase} one shows
  that the set $\{m_\lambda^-\,|\,\lambda\in \overline{P}^+_K\}$
  spans $A^{-W}_K$ as a vector space over $A^W_0$. Moreover,
  $s_{\alpha_1}\varpi_0=\varpi_0$ implies that $m^-_{K\varpi_0}=0$ and
  similarly $m^-_{K\varpi_1}=0$. Hence one gets the estimate
  \begin{align}\label{eq:first-est}
    \dim_{A^W_0}(A^{-W}_K)\le |\overline{P}^+_K|-2=K-1.
  \end{align}
  In particular, $A^{-W}_K=\{0\}$ for $K=0,1$. For $K\ge 2$
  the map $A^W_{K-2}\rightarrow A^{-W}_K$, $x\mapsto x\deltah_1$, is injective
  because $\overline{\C}[P]$ has no zero-divisors. Hence, by
  Proposition \ref{prop:AWKbase}, one has
  \begin{align*}
    \dim_{A^W_0}(A^{-W}_K)\ge \dim_{A^W_0}(A^{W}_{K-2})=K-1.
  \end{align*}
  for $K\ge 2$. In view of \eqref{eq:first-est} multiplication by
  $\deltah_1$ is an isomorphism $A^W_{K-2}\rightarrow A^{-W}_{K}$ of
  $A^W_0$-vector spaces.
\end{proof}
\begin{rema}\label{rem:deltah2}
  Let $\overline{\C}[2P]$ denote the completion of $\C[2P]$ in
  $\overline{\C}[P]$. The element $\deltah_2$ defined by \eqref{eq:deltah12} plays the same role in
  $\overline{\C}[2P]$ as the element $\deltah_1$ plays in
  $\overline{\C}[P]$. In particular one obtains that any element $x\in \overline{\C}[2P]\cap
  A^{-W}$ can be written as $x=\deltah_2y$ for a uniquely determined element $y\in A^W$. 
\end{rema}
We end this section with examples of elements in $A\setminus (A^W\cup A^{-W}\cup \C[P])$. Let $M\subset \Phi^+$ be a finite set of positive roots and define
\begin{align*}
  \deltah_{M}=\prod_{\alpha\in \Phi^+\setminus M}(1-e^{-\alpha}).
\end{align*}
For any $w\in W$ the set $w^{-1}(\Phi^+\setminus M)$ contains only finitely many negative roots. Hence
$\deltah_{M}\in w(\overline{\C}[P])$. This shows that $\deltah_{M}\in A$. This observation will be used in the proof of Lemma \ref{lem:inA2}. 
\subsection{The completion of the ring of fractions}\label{sec:frac-completion}
As we will see in Section \ref{sec:rad}, radial parts of elements of $U(\slzh)$ involve inverses of elements $1-e^{-\gamma}$ for some $\gamma\in \Phi^+$. The geometric series $\sum_{n=0}^\infty e^{-n\gamma}$ is an inverse of $1-e^{-\gamma}$ in $\overline{\C}[P]$, however, it is not contained in $A$ and hence we cannot act on it by $W$. On the other hand, let $\cR[P]$ denote the localization of $\C[P]$ with respect to the multiplicative set generated by all $(1-e^\alpha)$ for $\alpha\in \Phi^+$. For $\alpha\in \Phi^+$ one has
$(1-e^{-\alpha})(1-(1-e^{\alpha})^{-1})=1$ in $\cR[P]$ and hence $\cR[P]$ coincides with the localization of $\C[P]$ with respect to the multiplicative set generated by all $(1-e^{-\gamma})$ for $\gamma\in
\Phi^+$. This shows in particular that $W$ acts on $\cR[P]$. 

In this subsection, following \cite[Section 6]{a-EK95}, a completion of $\cR[P]$ is defined which contains $A$ and retains the action of $W$. To make this rigorous, for any $w\in W$, we define a new metric $d_w$ on $\C[P]$ by 
\begin{align*}
  d_w(a,b)=d_1(w^{-1}a,w^{-1}b).
\end{align*}
The completion of $\C[P_K]$ with respect to the metric $d_w$ is $w(\overline{\C}[P_K])$.
We now use the metrics $d_w$ to define a metric on $\cR[P]$. For every $w\in W$ there is an algebra homomorphism $\tau_w:\cR[P]\rightarrow w(\overline{\C}[P])$ determined by
\begin{align*}
  \tau_w(e^\lambda)&=e^\lambda &&\mbox{for all $\lambda\in P$,}\\
  \tau_w((1-e^\alpha)^{-1})&=\sum_{n\in \N_0} e^{n\alpha}&& \mbox{if
    $\alpha\in -w \Phi^+$,}\\
  \tau_w((1-e^\alpha)^{-1})&=-e^{-\alpha}\sum_{n\in \N_0} e^{-n\alpha}&& \mbox{if
    $\alpha\in w \Phi^+$.}
\end{align*}
In other words, $\tau_w$ maps any element to its expansion by geometric series for $(1-e^\gamma)^{-1}$ for all $\gamma\in -w\Phi^+$. For any $a,b\in \cR[P]$ define
\begin{align*}
  d(a,b) = \sum_{k\in \N} \frac{1}{2^k(k+1)} \sum_{w\in W, l(w)=k} \frac{d_w(\tau_w(a),\tau_w(b))}{1+d_w(\tau_w(a),\tau_w(b))}.
\end{align*}
One checks that $d$ defines a metric on $\cR[P]$. Let $\overline{\cR}[P]$ denote the completion of $\cR[P]$ with respect to the metric $d$.
\begin{rems}\label{rems:RP}
  1. In our case $W=\Z_2 \ast \Z_2$ and hence $W_k=\{w\in W\,|\,l(w)=k\}$ satisfies $|W_k|=k+1$. For more general roots systems the factor $k+1$ in the denominator should be replaced by $|W_k|$.

  2. A sequence $(r_n)$ in $\cR[P]$ is a Cauchy sequence with respect to $d$ if and only if for every $w\in W$ the sequence $\tau_w(r_n)$ is a Cauchy sequence in $w(\overline{\C}[P])$ with respect to $d_w$. Hence we can write informally
  \begin{align*}
  \overline{\cR}[P]=\big\{\sum_{n\in \N} a_n \,\big|\, a_n\in \cR[P],\,
  \sum_{n\in \N} \tau_w(a_n) \mbox{ converges in}
    \mbox{ $w(\overline{\C}[P])$ for every $w\in W$} \big\}.
\end{align*}
This is the definition given in \cite[p.~238]{a-EK95}, \cite[(6.2.7)]{phd-Kirillov95}. We prefer the formal definition of $\overline{\cR}[P]$ as a completion of a metric space because it allows us to consider Cauchy sequences in $\overline{\cR}[P]$ as elements in $\overline{\cR}[P]$, see for example Lemma \ref{lem:inA2}.

3. The multiplication map $\cR[P]\times \cR[P]\rightarrow \cR[P]$ is continuous and hence $\overline{\cR}[P]$ retains the algebra structure from $\cR[P]$.

4. For $a\in \C[P]$ one has $\tau_w(a)=a$. Hence the completion of $\C[P_K]$ with respect to the metric $d$ coincides with $A_K$. One obtains an injective algebra homomorphism $A\rightarrow \overline{\cR}[P]$. 

5.  The Weyl group action on $\cR[P]$ extends to an action on $\overline{\cR}[P]$ by algebra automorphisms. Indeed, one checks that for all $u,w\in W$ and $\gamma \in \Phi$ the relation
\begin{align*}
  \tau_w((1-e^{u\gamma})^{-1}) = u (\tau_{u^{-1}w}((1-e^{\gamma})^{-1}))
\end{align*} 
holds, and hence 
\begin{align*}
  d_w(\tau_w(ua), \tau_w(ub)) = d_{u^{-1}w}(\tau_{u^{-1}w}(a),\tau_{u^{-1}w}(b))
\end{align*}
for all $a,b\in \cR[P]$. By Remark 2.~above this shows that $(r_n)$ is a Cauchy sequence in $\cR[P]$ if and only if $(ur_n)$ is.

6. Let $\cR_2[P]$ denote the localization of $\C[P]$ with respect to the multiplicative set generated by all $(1-e^{-2\gamma})$ for $\gamma\in \Phi^+$. The material of this section is readily translated to obtain a completion $\overline{\cR}_2[P]$.  

7. The Weierstrass functions $\wp_{i,j}(e^{\alpha_1},e^{-\delta})$ for $i,j\in \{0,1\}$ defined in \eqref{eq:wp-formal} and \eqref{eq:p00}--\eqref{eq:p11} belong to $\overline{\cR}_2[P]$. To this end observe that $\frac{1}{1\pm e^{-\alpha}}=\frac{1\mp e^{-\alpha}}{1-e^{-2\alpha}}$. Moreover, $\wp_{ij}(e^{\alpha_1},e^{-\delta})$ are invariant under the action of $W$ on $\overline{\cR}_2[P]$ as follows from the relation $\frac{e^{-\beta}}{(1\pm e^{-\beta})^2}=\frac{e^{\beta}}{(1\pm e^{\beta})^2}$ which holds in $\cR_2[P]$ for all $\beta\in \Phi$.
\end{rems}
We end this subsection with a useful method to identify certain elements of $\overline{\cR}[P]$ as elements in $A$. Let $k\in \N_0$. We say that an element $a\in \cR[P]$ has poles of order at most $k$ if it is contained in the linear span 
\begin{align*}
  \Lin_\C\Big\{\frac{\C[P]}{\prod_{\gamma\in M}(1-e^{-\gamma})^{n_\gamma}}\,\Big|\, M\subset  \Phi^+ \mbox{ finite, } n_\gamma\le k\Big\}\subset \cR[P].
\end{align*}
\begin{lem}\label{lem:inA2}
  Let $k\in \N_0$ and let $(a_n)_{n\in \N}$ be a Cauchy sequence in $\cR[P]$ such that each $a_n$ has poles of order at most $k$. Then $(\deltah_1^k a_n)_{n\in \N}$ is a Cauchy sequence in $A$. The analogous statement holds for $\cR_2[P]$.
\end{lem}
\begin{proof}
  The observation at the end of Subsection \ref{sec:sym-theta} implies that $\deltah_1^k a_n\in A$ if $a_n$ has poles of order at most $k$. Moreover, as $(a_n)$ is a Cauchy sequence, so is $(\deltah_1^k a_n)$.
\end{proof}
\subsection{Twisted zonal spherical functions}\label{sec:twisted-zonal}
Recall that  $\overline{\C}[\SLzh]$ is a $U(\slzh)$-bimodule and let
$\lact$ and $\ract$ denote the left and right action of $U(\slzh)$ on
$\overline{\C}[\SLzh]$, respectively. Let $\eta, \chi$ be one-dimensional representations
of $\slzhth$. For $K\in \N_0$ define 
\begin{align}\label{eq:tw-zon-def}
  {}^\eta \overline{\C}[\SLzh]^\chi_K :=\{a\in
  \overline{\C}[\SLzh]_K\,|\, x\lact a=\eta(x)a \mbox{ and }& a\ract y=\chi(y)a\\
  \hfill
  &\mbox{ for all $x,y\in \slzhth$}\}.\nonumber
\end{align}
We call elements of $  {}^\eta \overline{\C}[\SLzh]^\chi_K$  twisted
zonal spherical functions of level $K$. Observe that for $\eta=\chi=0$ and $K=0$
the space   ${}^0 \overline{\C}[\SLzh]^0_0$ coincides with the
completion of $\oplus_{n\in \Z} \delta V(n\delta)\otimes V(n\delta)$
inside $\overline{\C}[\SLzh]_0$. For each $n\in \Z$ choose a basis
element $\varphi_{n\delta}\in \delta V(n\delta)\ot V(n\delta)$ such
that $\varphi_{n\delta}$ maps to $1$ under the canonical evaluation
map $\delta V(n\delta)\ot V(n\delta)\rightarrow \C$, $f\ot v\mapsto
f(v)$. Then $\varphi_{n\delta}\varphi_{m\delta}=\varphi_{(n+m)\delta}$ 
for all $n,m\in \Z$ and 
\begin{align}\label{eq:0C00}
  {}^0 \overline{\C}[\SLzh]^0_0=\Big\{\sum_{n\le
    n_0}a_n\varphi_{n\delta}\,|\, n_0\in \N, a_n\in \C\Big\}
\end{align}
is a field.
Observe that ${}^\eta \overline{\C}[\SLzh]^\chi_K$ is a vector space
over ${}^0 \overline{\C}[\SLzh]^0_0$ under multiplication. Recall the definition of $\oP^+_K$ from \eqref{eq:oP+K-def1}
and set
\begin{align}\label{eq:oP+K-def}
  \oP_K^+(\eta,\chi)=\{\lambda\in
  \oP^+_K\,|\,\lambda(h_j)-|\eta(B_j)|,\,
  \lambda(h_j)-|\chi(B_j)|\in 2\N_0 \mbox{ for $j=0,1$}\}. 
\end{align} 
Proposition \ref{prop:aff-inv} implies that for any $\lambda\in
\oP^+_K(\eta,\chi)$ there exists a nonzero element 
\begin{align*}
  \varphi_\lambda(\eta,\chi)\in \overline{\delta V(\lambda)}^\chi\ot
  \overline{V(\lambda)}^\eta 
\end{align*}
which is unique up to a scalar factor. Fix a highest weight vector
$v_\lambda\in V(\lambda)$ and let $v_\lambda^\ast\in \delta V(\lambda)$
denote the dual functional. We normalize
$\varphi_\lambda(\eta,\chi)$ such that the coefficient of
$v_\lambda^\ast\ot v_\lambda\in \delta V(\lambda)\ot V(\lambda)$ is
equal to one. The following proposition is an immediate consequence of
Proposition \ref{prop:aff-inv} and the construction of $\overline{\C}[\SLzh]_K$.
\begin{prop}\label{prop:basis}
  Let $K\in \N_0$ and let $\eta,\chi$ be one-dimensional representations of $\slzhth$.
  The set 
  \begin{align*}
      \{\varphi_\lambda(\eta,\chi)\,|\,\lambda\in \oP^+_K(\eta,\chi)\}
  \end{align*}
  is a basis of ${}^\eta \overline{\C}[\SLzh]^\chi_K$ over the field 
  ${}^0 \overline{\C}[\SLzh]^0_0$.
\end{prop}

\subsection{Generalized characters}\label{sec:gen-char}
Define a linear map  $\Psi: \C[\SLzh]\rightarrow \C[P]$ by
\begin{align*}
   \Psi(f_\mu\ot v_\nu)=f_\mu(v_\nu)e^\nu\quad \mbox{for all
     $f_\mu\in \delta V(\lambda)_\mu,\, v_\nu\in
     V(\lambda)_\nu,\,\lambda\in P^+$.}
\end{align*} 
By Remark \ref{rem:inA} the map $\Psi$ naturally extends to a linear map
  $\Psi:\overline{\C}[\SLzh]\rightarrow A$. Moreover, one has
  $\Psi(\overline{\C}[\SLzh]_K)\subset A_K$. By definition of the
  product on $\overline{\C}[\SLzh]$ the map $\Psi$ is a homomorphism of
  algebras. Following \cite[Section 7]{a-EK95}, \cite[6.3]{phd-Kirillov95} we call elements in the image of the map $\Psi$ generalized characters. 
\begin{lem}\label{lem:Psi}
  (1) The map $\Psi:{}^0\overline{\C}[\SLzh]^0_0\rightarrow A^{W}_0$ is an
  isomorphism of fields.\\
  (2)  The map $\Psi: \overline{\C}[\SLzh]\rightarrow A$ is
  $A^W_0$-linear under the identification of
  ${}^0\overline{\C}[\SLzh]^0_0$ with $A^{W}_0$ from (1).
\end{lem}
\begin{proof}
   Property (1) follows from Equation \eqref{eq:0C00} and Lemma
  \ref{lem:AW0}. Property (2) holds because $\Psi$ is an algebra
  homomorphism.
\end{proof}
\begin{rema}
  For general symmetrizable Kac-Moody algebras Mokler considered a restriction map to obtain a formal Chevalley restriction theorem, see \cite[Corollary 3.4]{a-Mokler03}. The map $\Psi$ is a special case of Mokler's restriction map.
\end{rema}
The image of zonal spherical functions $\varphi\in {}^0\overline{\C}[\SLzh]^0$ under the map $\Psi$ is
invariant under the action of the Weyl group. This holds more generally for twisted zonal spherical functions if the one-dimensional representations $\eta$ and $\chi$ coincide.
\begin{prop}\label{prop:invariant}
  Let $\chi$ be a one-dimensional representation of $\slzht$ and $\varphi\in {}^\chi\overline{\C}[\SLzh]_K^\chi$. Then $\Psi(\varphi)\in A_K^W$.
\end{prop}
\begin{proof}
  One may assume that $\varphi=\varphi_\lambda(\chi,\chi)$ for some
  $\lambda\in \oP^+_K(\chi, \chi)$. For symmetry reasons it suffices to show that
  $\Psi(\varphi)$ is invariant under the action of
  $s_{\alpha_1}$. Let $U_1\cong U(\slfrak_2)$ denote the subalgebra of
  of $U(\slzh)$ generated by $e_1,f_1,h_1$. Let moreover
  $U(\lfrak_1)$ denote the subalgebra generated by $e_1, f_1$, and $\hfrak$. Consider the
  decomposition
  \begin{align*}
     V(\lambda)=\bigoplus_m V_m
  \end{align*}
  of $V(\lambda)$ into an infinite sum of simple $U(\lfrak_1)$-
  modules. Each simple $U(\lfrak_1)$-module appears with finite multiplicity
  and is irreducible as a $U_1$-module. Then
  \begin{align*}
     \delta(V(\lambda))=\bigoplus_m V_m^\ast
  \end{align*}
  where $V_m^\ast$ denotes the linear dual space of $V_m$, and
  $V^\ast_m(V_n)=0$ for $m\neq n$. Write $\varphi=\sum
  \varphi_{m,n}$ with $\varphi_{m,n}\in V^\ast_m\ot V_n$. By
  definition one has $\Psi(\varphi)=\sum_m
  \Psi(\varphi_{m,m})$. The summands $\Psi(\varphi_{m,m})$ are twisted spherical
  functions for $U_1$ and hence $s_{\alpha_1}$-invariant, as may be
  checked by an $\slfrak_2$-calculation.  
\end{proof}
\begin{rema}\label{rema:notW}
  If $\eta\neq \chi$ and $\varphi\in {}^\eta\overline{\C}[\SLzh]^\chi$ then $\Psi(\varphi)$ is generally not 
  $W$-invariant.
\end{rema}
The image of ${}^\chi\overline{\C}[\SLzh]^\chi$ under the map
$\Psi$ is not all of $A^W$, even for $\chi=0$. By \eqref{eq:oP+K-def} one has   
  \begin{align*}
     \oP^+_K(0,0)=\begin{cases} 2 \oP^+_{K/2}&\mbox{if $K$ even,}\\ 0 &
       \mbox{if $K$ odd.}\end{cases}
  \end{align*}
  Let $A^{W,\theta}_K$ denote the $A^W_0$-vector subspace of $A^W_K$
  generated by all $m_\lambda$ with $\lambda\in \oP^+_K(0,0)$ and define 
  \begin{align*}
    A^{W,\theta}=\mathop{\oplus}_{K\in \N_0} A^{W,\theta}_K.
  \end{align*}
  Observe that $A^{W,\theta}= \mathop{\oplus}_{\lambda\in
    2\N_0 \varpi_0+2\N_0\varpi_1} A^W_0 m_\lambda$ is a subalgebra of $A^W$.
  Finally, to shorten notation, write $\varphi_\lambda=\varphi_\lambda(0,0)$.  
\begin{prop}
  Let $K\in \N_0$. The map $\Psi:{}^0\overline{\C}[\SLzh]^0_K\rightarrow
  A^{W,\theta}_K$ is an isomorphism of $A^W_0$-vector spaces under the
  identification of ${}^0\overline{\C}[\SLzh]^0_0$ with $A^{W}_0$ from
  Lemma \ref{lem:Psi}.(1). 
\end{prop}
\begin{proof}
 One may assume that $K$ is even. By Proposition \ref{prop:invariant}
 and the remarks at the beginning of this subsection one has
 $\Psi({}^0\overline{\C}[\SLzh]^0_K)\subset A^W_K$. Moreover, for any
 $\lambda \in  \oP^+_K(0,0)$ all weights of $V(\lambda)$ belong to
 $2P+\Z \delta$. This implies that $\Psi({}^0\overline{\C}[\SLzh]^0_K)\subseteq
 A^{W,\theta}_K$. By Proposition \ref{prop:basis} one has
 $\dim_{A^W_0}({}^0\overline{\C}[\SLzh]^0_K)=|\oP^+_K(0,0)|=
 \dim_{A^W_0}(A^{W,\theta}_K)$. Hence it suffices to show that
 $\Psi:{}^0\overline{\C}[\SLzh]^0_K\rightarrow A^{W,\theta}_K$ is
 injective. 
 
To this end assume that there exists a non-empty subset
$\tilde{P}^+_K\subseteq \oP^+_K(0,0)$ and nonzero elements
$a_\lambda\in A^W_0$ such that $\sum_{\lambda\in \tilde{P}^+_K} a_\lambda
\Psi(\varphi_\lambda)=0$. For every $\lambda\in \tilde{P}^+_K$ there
is a uniquely determined integer $n_\lambda\in \Z$  such that 
\begin{align*}
  a_\lambda=e^{n_\lambda \delta}(a_0^\lambda + \sum_{n<0} a^\lambda_n e^{n\delta})
\end{align*}  
where $a_n^\lambda\in \C$ for all $n\in -\N_0$ and $a^\lambda_0\neq
0$. Set $n^{max}=\max\{n_\lambda\,|\,\lambda\in
\tilde{P}^+_K\}$. Observe that the weights in $\oP^+_K(0,0)$ are
linearly ordered with respect to the partial ordering $\le$ and that
they only differ by multiples of $\alpha_1$. Choose $\mu\in
\tilde{P}^+_K$ maximal such that $n_\mu=n^{\max}$. Then
$\sum_{\lambda\in \tilde{P}^+_K} a_\lambda \Psi(\varphi_\lambda)$ has
$e^{\mu+ n^{max}\delta}$-coefficient $a_0^\mu\neq 0$. This is
a contradiction. 
\end{proof}
We now turn to the case of general one-dimensional representation $\eta, \chi$. 
Despite Remark \ref{rema:notW} twisted zonal spherical functions for $\eta\neq \chi$ can also be used to construct a basis of $A^{W, \theta}_K$. To obtain $W$-invariance one needs to normalize. Define
\begin{align*}
  \oP^+(\eta,\chi)=\bigcup_{K\in \N_0} \oP^+_K(\eta,\chi).
\end{align*}
To shorten notation set $m_j=\max(|\eta(B_j)|,|\chi(B_j)|)$, and define
$\lambda_0(\eta,\chi)=m_0 \varpi_0+m_1\varpi_1$. Observe that
\begin{align}\label{eq:oP-rel}
  \oP^+(\eta,\chi)=\begin{cases} \lambda_0(\eta,\chi)+\oP^+(0,0)&
    \mbox{if $\eta(B_j)-\chi(B_j)\in 2\Z$ for $j=0,1$,}\\
        0&\mbox{else.}  \end{cases}
\end{align}
\begin{prop}
  Let $\eta, \chi$ be one-dimensional representations of $\slzhth$ such that
  $\eta(B_j)-\chi(B_j)\in 2\Z$ for $j=0,1$. The map
  \begin{align}\label{eq:map}
      {}^0 \overline{\C}[\SLzh]^0\rightarrow {}^\eta
     \overline{\C}[\SLzh]^\chi, \qquad f\mapsto \varphi_{\lambda_0(\eta,\chi)}f
  \end{align}
  is an isomorphism of ${}^0 \overline{\C}[\SLzh]^0_0$-vector spaces.
\end{prop}
\begin{proof}
  Set $\Kfrak={}^0 \overline{\C}[\SLzh]^0_0$. The map \eqref{eq:map}
  is $\Kfrak$-linear and injective. Moreover, it maps ${}^0
  \overline{\C}[\SLzh]^0_K$ to ${}^\eta
  \overline{\C}[\SLzh]^\chi_{K+m_0+m_1}$. By Proposition \ref{prop:basis} and Equation
  \eqref{eq:oP-rel} one has
  \begin{align*}
    \dim_\Kfrak({}^0
    \overline{\C}[\SLzh]^0_K)=|\oP^+_K(0,0)|=|\oP^+_K(\eta,\chi)|=\dim_\Kfrak({}^\eta
    \overline{\C}[\SLzh]^\chi_{K+m_0+m_1}) 
  \end{align*}
  which proves the proposition.
\end{proof}
As $\Psi$ is an algebra homomorphism one obtains the following consequence.
\begin{cor}\label{cor:basisSymTheta}
  Let $\eta, \chi$ be one-dimensional representations of $\slzhth$ such that
  $\eta(B_j)-\chi(B_j)\in 2\Z$ for $j=0,1$. As before set
  $m_j=\max(|\eta(B_j)|,|\chi(B_j)|)$ for $j=0,1$ and let $K\in
  \N_0$. 
  \begin{enumerate}
    \item Let $f\in  {}^\eta
      \overline{\C}[\SLzh]^\chi_{K+m_0+m_1}$. Then $\Psi(f)$ is
      divisible by 
      $\Psi(\varphi_{\lambda_0(\eta,\chi)})$ and
      $\Psi(f)/\Psi(\varphi_{\lambda_0(\eta,\chi)})\in A^{W,\theta}_K$.  
    \item The elements 
  \begin{align*}
     \Big \{\frac{\Psi(\varphi_\lambda)}{\Psi(\varphi_{\lambda_0(\eta,\chi)})}\,\Big|\,
     \lambda\in \oP^+_K(\eta,\chi) \Big\}
  \end{align*}
  form a basis of $A^{W,\theta}_K$ as a vector space over $A^W_0$.
  \end{enumerate}
\end{cor}
\section{The radial part of the Casimir element}\label{sec:rad}
Twisted zonal spherical functions for $\slzhth$ can be interpreted as eigenfunctions for the radial part of the Casimir element of $U(\slzh)$. In Subsections \ref{sec:inf-Cartan} and \ref{sec:rad-comp} we translate the construction of radial parts from \cite[Sections 2, 3]{a-CM82} to the infinite dimensional setting of $\slzhth$. In Subsections \ref{sec:Cas-act} and \ref{sec:Cas-non-triv} we calculate the radial part of the Casimir element explicitly. This is done in two steps, first considering the case of trivial one-dimensional representations of $\slzhth$ and then calculating additional terms which appear in the general case. After conjugation by a suitable element $\deltah\in \overline{\C}[P]$ the resulting formal differential operator on $\overline{\C}[P]$ can be expressed in terms of Weierstrass $\wp$-functions. 

\subsection{The infinitesimal Cartan decomposition}\label{sec:inf-Cartan}
Let $H=(\C\setminus\{0\})^3$ be a three dimensional complex torus. The adjoint
action of $\hfrak$ on $\slzh$ lifts to the adjoint action $\Ad$ of $H$
on $\slzh$ by Lie algebra automorphisms. More explicitly, for
$\alpha=n_0\alpha_0+n_1\alpha_1\in Q$ and $a=(t_0,t_1,t_d)\in H$ define 
\begin{align}\label{eq:alpha(a)}
  \alpha(a)=(t_0t_1^{-1})^{2(n_0-n_1)}t_d^{n_0}.
\end{align}
With this notation one has for any $X\in (\slzh)_\alpha$ the relation
\begin{align*}
  \Ad(a)(X)=\alpha(a)X.
\end{align*}
Recall that $\Phi$ denotes the root system of $\slzh$ and define 
\begin{align*}
  H_{\mathrm{reg}}=\{a\in H\,|\,\alpha(a)\neq 1\quad\mbox{for all
    $\alpha\in \Phi$}\}.
\end{align*}
\begin{rema}\label{rem:inj}
 Observe that the set $H_{\mathrm{reg}}$ is fairly big: it consists of
 all those triples $(t_0,t_1,t_d)$ such that $(t_d)^n\notin
 \{1,t_0^2t_1^{-2}, t_0^{-2}t_1^2\}$ for all $n\in \N$. 
\end{rema}
For any $a\in H$ and $X\in U(\slzh)$ define
$X^a=\Ad(a^{-1})(X)$. Moreover, define
\begin{align*}
  \msA=U(\hfrak)\ot U(\slzhth)\ot U(\slzhth)
\end{align*}
which we consider only as a vector space. The following proposition is
obtained in the same way as \cite[Theorem 2.1]{a-CM82}.
\begin{prop}\label{prop:CM1}
  Let $a\in H_{\mathrm{reg}}$. The linear map $\Gamma_a:\msA\rightarrow
  U(\slzh)$ defined by 
  \begin{align*}
    \Gamma_a(H\ot X\ot Y)=X^aHY
  \end{align*}
  is a linear isomorphism.
\end{prop}
\begin{proof}
  For any $\gamma \in \Phi^+$ define $z_\gamma\in (\slzh)_\gamma$ by
  \begin{align*}
    z_\gamma=\begin{cases}
               t^n\ot e & \mbox{if $\gamma=n\delta + \alpha_1$ for some $n\ge 0$,}\\
               t^n\ot f & \mbox{if $\gamma=n\delta - \alpha_1$ for some $n>0$,}\\
               t^n\ot h & \mbox{if $\gamma=n\delta$ for some $n>0$.}
             \end{cases}
  \end{align*}
  As pointed out in Remark \ref{rema:Onsager} one has
  \begin{align*}
    \slzhth=\bigoplus_{\gamma\in \Phi^+} \C(z_\gamma + \thetah(z_\gamma))
  \end{align*}
  and for any $a\in H$ one hence one gets
  \begin{align*}
    (\slzhth)^a=\bigoplus_{\gamma\in \Phi^+} \C(\gamma(a)^{-1} z_\gamma + \gamma(a) \thetah(z_\gamma)).
  \end{align*}
  If $a\in H_{\mathrm{reg}}$ then the triangular decomposition $\slzh=\nfrak^+\oplus \hfrak\oplus \nfrak^-$ implies that $\slzh=\slzhth\oplus \hfrak\oplus (\slzhth)^a$.
  Now the proposition follows from the Poincar\'e-Birkhoff-Witt Theorem. 
\end{proof}
Consider the group algebra of the root lattice $\C[Q]$. In analogy to Remark \ref{rems:RP}.6
let $\cR_2[Q]$ denote the localization of $\C[Q]$ with respect
to the multiplicative set generated by all $(1-e^{2\gamma})$ for
$\gamma\in \Phi^-$. In the following, for simplicity, we write 
\begin{align*} 
  \cR=\cR_2[Q].
\end{align*}  
For any $X\in U(\slzh)$ the preimage $\Gamma_a^{-1}(X)$ depends on $a\in H_{\mathrm{reg}}$. This dependence, however, can be expressed in terms of coefficients in $\cR$ independently of $a$.
To this end let $\cF(H_{\mathrm{reg}})$ denote the algebra of functions on
$H_{\mathrm{reg}}$. There is an algebra homomorphism $\cR\rightarrow
\cF(H_{\mathrm{reg}})$ such that $e^\alpha$ maps to the function
given by \eqref{eq:alpha(a)} for all $\alpha\in Q$. By Remark
\ref{rem:inj} this algebra homomorphism is injective. 
Extend the map $\Gamma_a$ from Proposition \ref{prop:CM1} linearly to
$\cR\ot \msA$ by 
\begin{align*}
  f\ot X\mapsto f(a)\Gamma_a(X)\qquad \mbox{for all $f\in \cR$,  $X\in
    \msA$.}
\end{align*}
We denote this map by $\Gamma_a$, too. In complete analogy to
\cite[Theorem 2.4]{a-CM82} one obtains the following result.
\begin{thm}\label{thm:GamPiX}
  For each $X\in U(\slzh)$ there exists a unique $\Pi(X)\in \cR\ot
  \msA$ such that $\Gamma_a(\Pi(X))=X$ for every $a\in H_{\mathrm{reg}}$.
\end{thm}
\subsection{Radial components}\label{sec:rad-comp}
To describe radial parts of elements in $U(\slzh)$ one needs an action
of a crossed product $\cR\rtimes U(\hfrak)$ on $\overline{\C}[P]$.
To this end observe first that the completion $\overline{\C}[P]$ is a
left module algebra over the Hopf algebra $U(\hfrak)$ via the action
given by
\begin{align}\label{eq:h-act}
  h\lact (\sum a_\mu e^\mu)= \sum a_\mu \mu(h)e^{\mu}\qquad \mbox{for
    $h\in \hfrak$}
\end{align}
The subalgebra $\C[Q]$ is invariant under this action. Moreover, \eqref{eq:h-act} extends uniquely to an action of $U(\hfrak)$ on the localization $\cR$
\cite[Lemma 15.1.23]{b-McCoRo}. Hence one may form the crossed product
$\cR\rtimes U(\hfrak)$. Moreover, analogously to $\tau_1:\cR[P]\rightarrow \overline{\C}[P]$ defined in Section \ref{sec:frac-completion}, there is an algebra homomorphism
$\otau_1:\cR\rightarrow \overline{\C}[P]$ uniquely determined by
\begin{equation}
  \label{eq:tau-def}
  \begin{aligned}
    \otau_1(e^{\alpha_i})&=e^{\alpha_i},&\mbox{for $i=0,1$,}\\
    \otau_1((1-e^{2\gamma})^{-1})&=\sum_{k=0}^\infty e^{2k\gamma}&\mbox{for
      $\gamma\in \Phi^-$}.
  \end{aligned}  
\end{equation}
The desired left  $\cR\rtimes U(\hfrak)$-module structure
on $\overline{\C}[P]$ is defined by
\begin{align*}
  (f\ot x)\lact \phi&=\otau_1(f)(x\lact \phi)&\mbox{for $f\in \cR$, $x\in
    U(\hfrak)$, $\phi\in \overline{\C}[P]$.}
\end{align*}

Now let $\chi$, $\eta$ be one-dimensional representations of $\slzhth$. Consider the linear
map $\id\ot\id\ot \chi\ot \eta:\cR\ot \msA\rightarrow\cR\ot U(\hfrak)$
defined by applying $\chi$ and $\eta$ to the last two tensor factors of
$\msA=U(\hfrak)\ot U(\slzhth)\ot U(\slzhth)$. Define a linear map
\begin{align*}
  \Pi_{\eta,\chi}:U(\slzh)\rightarrow \cR\ot U(\hfrak),\quad
  \Pi_{\eta, \chi}(X)=(\id\ot\id\ot \chi\ot \eta) \circ\Pi(X).
\end{align*}
We call $\Pi_{\eta,\chi}(X)$ the $(\eta,\chi)$-radial part of the element $X\in
U(\slzh)$. For $\eta=\chi=0$ we call $\Pi_{0,0}(X)$ the radial part of $X$. The significance of $(\eta,\chi)$-radial parts stems from the following result.
\begin{thm}\label{thm:rad-part}
  Let $\eta,\chi$ be one-dimensional representations of $\slzhth$ and let
  $\varphi\in {}^\eta\overline{\C}[\SLzh]^\chi$. Then
  \begin{align*}
    \Psi(X\lact \varphi) = \Pi_{\eta,\chi}(X)\lact \Psi(\varphi)\qquad
    \mbox{for all $X\in U(\slzh)$.}
  \end{align*}  
\end{thm}
In the proof of Theorem \ref{thm:rad-part} below we will approximate
elements in $\overline{V}(\mu)$, $\mu\in P^+$, by elements 
of $V(\mu)$. The following lemma describes the behavior of the
error term of such an approximation under the action of $U(\slzh)$. For
any subset $M\subseteq P$ and any $U(\slzh)$-module $V$ we write
$V_M$ to denote the linear span of the weight spaces $V_\lambda$ for
$\lambda\in M$. Recall that $\Phi=(\Z \delta +\Phi_1)\setminus\{0\}$ where $\Phi_1=\{-\alpha_1,0,\alpha_1\}$.
\begin{lem}\label{lem:approx}
  Let $X\in U(\slzh)$, $\mu\in P^+$, and $F\in
  \overline{V}(\mu)$. Write $F=\sum_{m=0}^\infty f_m$ with $f_m\in
  V(\mu)_{\mu-m\delta+\Phi_1}$, and for any $n\in \N_0$ define
  $F_n=\sum_{m=0}^n f_m$ . Then there exists $N\in \N$ such that for
  all $n\in \N_0$ one has
  \begin{align*}
    X\lact (F-F_n)\in \prod_{k>n-N}V(\mu)_{\mu-k\delta+\Phi_1}.
  \end{align*}
\end{lem}
\begin{proof}
  This claim holds because the action of $X$ shifts weights only by
  the elements of a finite subset of $Q$.
\end{proof}
For any $h=\nu_0 h_0+\nu_1 h_1 + \nu_d d \in \hfrak$ define $a(h)\in (\C\setminus\{0\})^3$ by
\begin{align}\label{eq:a(nu)}
  a(h)=(e^{\nu_0},e^{\nu_1},e^{\nu_\delta}).
\end{align}
Define, moreover, 
\begin{align*}
  \hfrak_{\mathrm{reg}}=\{\nu_0 h_0+\nu_1 h_1 + \nu_d d\in \hfrak \,|\, n\nu_d\notin\{0,\pm2(\nu_0-\nu_1)\}+2\pi i \Z \mbox{ for all $n\in \N$}\}.
\end{align*}
By Remark \ref{rem:inj} one has $a(h)\in H_{\mathrm{reg}}$ if and
only if $h\in \hfrak_{\mathrm{reg}}$. Moreover, for any $\alpha=n_0\alpha_0+n_1\alpha_1\in Q$ one has by \eqref{eq:alpha(a)} the relation
\begin{align*}
  \alpha(a(h))= e^{2(\nu_0-\nu_1)(n_0-n_1)+\nu_d n_0}=e^{\alpha(h)}.
\end{align*} 
Any $f=\sum a_\mu e^\mu\in \C[P]$ may be considered as a function on
$\hfrak$ by 
\begin{align}\label{eq:f(nu)}
  f(h) =\sum a_\mu e^{\mu(h)}\qquad\mbox{for all $h\in \hfrak$.}
\end{align}
Definitions \eqref{eq:a(nu)} and \eqref{eq:f(nu)} are compatible via
the relation
\begin{align*}
  e^\alpha(h)=\alpha(a(h))\qquad \mbox{for all $\alpha\in Q$, $h\in \hfrak$.}
\end{align*}
\begin{lem}\label{lem:left-right}
  Let $F\in \C[\SLzh]$ and $X\in U(\slzh)$. Then
  \begin{align}\label{eq:left-right}
    \Psi(X^{a(h)}\lact F)(h) = \Psi(F\ract X)(h)\qquad
    \mbox{for all $h\in \hfrak$.}
  \end{align}
\end{lem}
\begin{proof}
  By linearity we may assume that $X\in U(\slzh)_\alpha$ for some
  $\alpha\in Q$ and that $F=f_\gamma\ot v_\beta$ for some 
  $f_\gamma\in \delta V(\lambda)_\gamma$, $v_\beta\in
  V(\lambda)_\beta$ where $\lambda\in P^+$ and $\gamma,\beta\in P$. We
  may moreover assume that $\gamma=\beta+\alpha$ because otherwise both sides
  of \eqref{eq:left-right} vanish. For any $h\in \hfrak$ one
  obtains
  \begin{align*}
     e^{-\alpha(h)} \Psi(X\lact F)(h) &=  e^{-\alpha(h)}
     \Psi(f_\gamma\ot X v_\beta)(h) =  e^{\gamma(h)-\alpha(h)}
     f_\gamma(Xv_\beta)\\
      & = \Psi(f_\gamma X\ot v_\beta)(h)=\Psi(F\ract X)(h).
  \end{align*}  
  In view of $X^{a(h)}=e^{-\alpha(h)}X$ the above relation
  proves the lemma. 
\end{proof}
With the above preparations we are now ready to prove the statement about radial parts. 

\medskip
\noindent \textit{Proof of Theorem \ref{thm:rad-part}.}
Let $h\in \hfrak_{\mathrm{reg}}$. By Theorem \ref{thm:GamPiX} one has
$X=\Gamma_{a(h)}(\Pi(X))$. Write
\begin{align}\label{eq:Pi(X)}
  \Pi(X)=\sum_j r_j\ot A_{1j}\ot A_{2j} \ot A_{3_j}\in \cR\ot
  U(\hfrak)\ot U(\slzhth) \ot U(\slzhth)
\end{align}
and chose $p\in \C[Q]$ such that $p^{-1} \in \cR$ and $pr_j\in \C[Q]$
for all $j$ to clear denominators. For any $F\in \C[\SLzh]$ one
obtains by Theorem \ref{thm:GamPiX} the relation  
\begin{align*}
  (p \Psi(X\lact F))(h)&= (p\Psi(\Gamma_{a(h)}(\Pi(X))\lact
  F))(h)
\end{align*}
which by Equation \eqref{eq:Pi(X)} gives
\begin{align*}
  (p \Psi(X\lact F))(h) &= \sum_j (pr_j)(h) \Psi(A_{2j}^{a(h)}A_{1j}A_{3j}\lact
   F)(h).
\end{align*}
Now Lemma \ref{lem:left-right} implies
\begin{align*}
 (p \Psi(X\lact F))(h)  &=\sum_j
 (pr_j)(h)\Psi(A_{1j}A_{3j}\lact F \ract A_{2j})(h) \qquad \mbox{ for any $h\in \hfrak_{\mathrm{reg}}$}
\end{align*}
and hence
\begin{align}\label{eq:pPsiXF}
 p \Psi(X\lact F)&=\sum_j (pr_j\ot A_{1j})\lact \Psi(A_{3j}\lact F \ract A_{2j}).
\end{align}
Lemma \ref{lem:approx} now implies that relation \eqref{eq:pPsiXF}
also holds for any $F\in \overline{\C}[\SLzh]$. In particular, for any
$\varphi\in {}^\eta\overline{\C}[\SLzh]^\chi$ one obtains
\begin{align*}
   p \Psi(X\lact F)&=\sum_j \eta(A_{3j})\chi(A_{2j})(pr_j\ot
   A_{1j})\lact \Psi(F).
\end{align*}
Multiplication by $p^{-1}$ gives  $\Psi(X\lact \varphi) =
\Pi_{\eta,\chi}(X)\lact \Psi(\varphi)$. 
\hfill $\square$

\subsection{Action of the Casimir}\label{sec:Cas-act}
We recall the definition of the Casimir element for $\slzh$. Let
$\{u_j\}$ be a basis of $\slfrak_2$ and  $\{u^j\}$ the dual basis with
respect to the nondegenerate invariant form $(\cdot,\cdot)$ from
Section \ref{sec:slzh}. By \cite[12.8.3]{b-Kac1} the Casimir element of
$\slzh$ is given by
\begin{align}\label{eq:Casimir}
  \Omega=2(c+h^\vee)d + \mathring{\Omega} + 2\sum_{n=1}^\infty\sum_i u_i^{(-n)}u^{i(n)}
\end{align}
as an element of the completion $\widehat{U}(\slzh)$ of $U(\slzh)$ described in Example \ref{eg:hayashi}. Here $h^\vee=2$ is the dual Coxeter number for $\slfrak_2$,
\begin{align}\label{eq:Omega-fin}
  \mathring{\Omega}=\frac{1}{2} h_1^2-h_1+2 e_1f_1\in  U(\slfrak_2)
\end{align}
is the Casimir element for $\slfrak_2$, and $u^{(n)}= t^{n}\ot u$ for
all $n\in \Z$, $u\in \slfrak_2$. The Casimir element $\Omega$ acts on
any highest weight module with highest weight $\lambda\in P$ by
the scalar factor $(\lambda,\lambda+2\rho)$ where $\rho=\varpi_0+\varpi_1$.
The action of $U(\slzh)$ on $\overline{V}(\lambda)$ extends by continuity to an action of the completion $\widehat{U}(\slzh)$ on $\overline{V}(\lambda)$.

In \cite[Example 2.7]{a-CM82} the radial part of $\mathring{\Omega}$
was calculated to be
\begin{align}\label{eq:PiOmega-fin}
  &\Pi(\mathring{\Omega})=\frac{1}{2} \ot h_1^2\ot 1\ot 1
  -\frac{1+e^{2\alpha_1}}{1-e^{2\alpha_1}}\ot h_1\ot 1\ot 1\\
  +2&\frac{e^{2\alpha_1}}{(1-e^{2\alpha_1})^2}\ot  1\ot (X^2\ot 1 +1\ot X^2)
   -2\frac{e^{\alpha_1}(1+e^{2\alpha_1})}{(1-e^{2\alpha_1})^2}\ot 1\ot (X\ot
   X). \nonumber
\end{align}
where $X=f_1-e_1$.
In a similar way the radial parts $\Pi(u_i^{(-n)}u^{i(n)})$ for $n\in
\N$ are determined by means of the next proposition.
\begin{prop}\label{prop:radYtY}
  Let $\alpha\in \Phi^-$, $Y\in (\slzh)_\alpha$, and $X=Y+\thetah(Y)$.
  Then
\begin{align}
  \Pi(Y\thetah(Y))=& - \frac{e^{2\alpha}}{(1-e^{2\alpha})}\ot [Y,\thetah(Y)]\ot 1\ot 1
    +\frac{e^\alpha(1+e^{2 \alpha})}{(1-e^{2\alpha})^2}\ot 1\ot X\ot X \nonumber\\
   &\quad -\frac{e^{2\alpha}}{(1-e^{2\alpha})^2}\ot
  1\ot (X^2\ot 1 +1\ot X^2).\label{eq:Pi(YtY)}
\end{align}
\end{prop}
\begin{proof}
  For $\alpha\in \Phi^-$ one can express any $Y\in(\slzh)_\alpha$ in
  terms of $X=Y+\thetah(Y)$ by 
\begin{align*}
 Y&= \frac{\alpha(a)}{1-\alpha(a)^2}(X^a-\alpha(a)X),\\
 \thetah(Y)&= \frac{\alpha(a)^{-1}}{1-\alpha(a)^{-2}}(X^a-\alpha(a)^{-1}X)
\end{align*}
for any $a\in H_{\mathrm{reg}}$ \cite[Lemma 2.2]{a-CM82}.
Hence one has 
  \begin{align*}
    Y\thetah(Y)&=\frac{-\alpha(a)^2}{(1-\alpha(a)^2)^2}(X^a-\alpha(a)X)
    (X^a-\alpha(a)^{-1}X)\\  
    &=\frac{-\alpha(a)^2}{(1-\alpha(a)^2)^2}((X^a)^2+
    X^2-(\alpha(a)+\alpha(a)^{-1})X^aX+\alpha(a)[X^a,X]).
  \end{align*}
  As $[X^a,X]=-(\alpha(a)-\alpha(a)^{-1})[Y,\thetah(Y)]$ one obtains
  \begin{align*}
    Y\thetah(Y) &=\frac{-\alpha(a)^2}{(1{-}\alpha(a)^2)}[Y,\thetah(Y)]
    + \frac{\alpha(a)(1{+}\alpha(a)^2)}{(1{-}\alpha(a)^2)^2}X^a X 
      -\frac{\alpha(a)^2}{(1{-}\alpha(a)^2)^2}((X^a)^2{+}X^2). 
  \end{align*}  
  As $[Y,\thetah(Y)]\in \hfrak$ Equation \eqref{eq:Pi(YtY)} holds by
  definition of $\Pi$ in Theorem \ref{thm:GamPiX}.
\end{proof}
  Let $\eta, \chi$ be one-dimensional representations of $\slzhth$. As $\Omega$ is contained in the completion $\widehat{U}(\slzh)$, the radial part of $\Omega$ is only defined formally as
   \begin{align*}
     \Pi_{\eta,\chi}(\Omega)=\Pi_{\eta, \chi}(2(c+h^\vee)d +\mathring{\Omega}) + 2\sum_{n=1}^\infty\sum_i \Pi_{\eta, \chi}(u_i^{(-n)}u^{i(n)}).
   \end{align*}
We will see in Subsection \ref{sec:Cas-non-triv} that $\Pi_{\eta,\chi}(\Omega)$ converges in $\overline{\cR}_2[P]\ot U(\hfrak)$. Hence $\Pi_{\eta, \chi}$ is a well defined operator on $\overline{\C}[P]$ and Theorem \ref{thm:rad-part} implies that twisted zonal spherical functions provide eigenfunctions for this operator via $\Psi$. For now we restrict to the case $\chi=\eta=0$ where only the first summand in \eqref{eq:Pi(YtY)} contributes to the radial part. 

For $n\in\N$ one has 
\begin{align}\label{eq:sum-uu}
  \sum_i u_i^{(-n)}u^{i(n)}&= e^{(-n)}f^{(n)}+\frac{1}{2}h^{(-n)}
  h^{(n)}+ f^{(-n)} e^{(n)}\\
  &=-e^{(-n)}\thetah(e^{(-n)})-\frac{1}{2}h^{(-n)}
  \thetah(h^{(-n)})- f^{(-n)} \thetah(f^{(-n)}).\nonumber
\end{align}
For each of the three summands one can calculate the radial part by
means of the above proposition. Using moreover the relations
\begin{align*}
  [f^{(-n)},e^{(n)}]=-h_1- n c,\quad  [e^{(-n)},f^{(n)}]=h_1 -
  n c, \quad [h^{(-n)},h^{(n)}]=-2 nc
\end{align*}
one obtains for $\eta=\chi=0$ the formula
\begin{align}\label{eq:radial-part}
  \Pi_{0,0}(\Omega)=& 2\ot (c+ h^\vee )d + \frac{1}{2}\ot h_1^2 +
  \frac{1+e^{-2\alpha_1}}{1-e^{-2\alpha_1}}\ot h_1 \\&
      + 2\sum_{n=1}^\infty\Bigg(
      \frac{e^{-2n\delta-2\alpha_1}}{1-e^{-2n\delta-2\alpha_1}}-\frac{e^{-2n\delta
          +2\alpha_1}}{1-e^{-2n\delta+2\alpha_1}}\Bigg)\ot h_1 \nonumber\\ 
      &+2\sum_{n=1}^\infty\Bigg(\frac{n   e^{-2n\delta+2\alpha_1}}{1-e^{-2n\delta+2\alpha_1}}+\frac{ne^{-2n\delta}}{1-e^{-2n\delta}}
      +\frac{ne^{-2n\delta+2\alpha_1}}{1- e^{-2n\delta-2\alpha_1}}
      \Bigg)\ot c.\nonumber
\end{align}
The above formula shows that $\Pi_{0,0}(\Omega)$ converges in $\overline{\cR}_2[P]\ot U(\hfrak)$.
Formula \eqref{eq:radial-part} may be rewritten with a summation
over all positive roots.
For any $\alpha\in\hfrak^\ast$ write $\partial_\alpha$
to denote the corresponding element in $\hfrak$ under the
identification of $\hfrak$ and $\hfrak^\ast$ given by the invariant
bilinear form. By \eqref{eq:h=hast} one has
\begin{align*}
  \partial_\rho=h_1/2 +2 d,\quad \partial_{\alpha_1}=h_1,\quad \partial_\delta=c. 
\end{align*}
Moreover, define $\Delta=2cd+h_1^2/2$ and observe that $\Delta\in U(\hfrak)$ is the
unique element with
\begin{align}\label{eq:Laplace}
  \Delta\lact e^{\lambda}=(\lambda,\lambda) e^\lambda.
\end{align}
With the above formulas, suppressing tensor signs, and using $h^\vee=2$
one obtains
\begin{align}\label{eq:radial-part2}
   \Pi_{0,0}(\Omega) =&  \Delta + 2 \partial_\rho +
   2\sum_{\alpha\in \Phi^+}
   \frac{e^{-2\alpha}}{1-e^{-2\alpha}}\, \partial_\alpha. 
\end{align}
In analogy to \cite[Theorem 6.1]{a-EK95} we want to eliminate the first
order part of $\Pi_{0,0}(\Omega)$ by conjugation by a suitable element. To this end define
\begin{align}\label{eq:deltah-def}
  \deltah &=e^\rho \prod_{\alpha\in \Phi^+} \exp\left( \frac{1}{2}\ln(1{-}e^{-2\alpha})\right)
          =e^\rho \exp\left(-\frac{1}{2}\sum_{\alpha\in \Phi^+} \sum_{n\in \N} \frac{e^{-2n\alpha}}{n}\right) \in \overline{\C}[P].  
\end{align}
Informally we can write $\hat{\delta}= e^{\rho}\prod_{\alpha\in \Phi^+}(1-e^{-2\alpha})^{1/2}$, and by construction one has
\begin{align}\label{eq:delta2delta}
  \deltah^2=\deltah_2.
\end{align}
The radial part $\Pi_{0,0}(\Omega)$ may also be considered as an element of $\overline{\C}[P]\ot U(\hfrak)$ via the map $\otau_1$ given by \eqref{eq:tau-def}. Recall the definition of $\wp_{1,1}(y,q)$ from Appendix \ref{ap:theta}.
\begin{prop}\label{prop:Om-conj}
 The following relations hold in $\overline{\C}[P]\rtimes U(\hfrak)$:
 \begin{align*} 
   \hat{\delta}\Pi_{0,0}(\Omega) \hat{\delta}^{-1}&= \Delta-
   (\rho,\rho)+  \sum_{\alpha\in \Phi^+}
   \frac{e^{-2\alpha}}{(1-e^{-2\alpha})^2} (\alpha,\alpha)\\
    &=\Delta- (\rho,\rho) + \frac{1}{2\pi^2} \wp_{1,1}(e^{2\alpha_1},e^{-\delta}). 
 \end{align*}
\end{prop}
\begin{proof}
  The second equality follows immediately from $(n\delta \pm
  \alpha_1,n\delta\pm \alpha_1)=(\alpha_1,\alpha_1)=2$, from $(\delta,\delta)=0$, and from
  formula \eqref{eq:p11}.
  To verify the first equality we follow the argument given in the
  proof of \cite[Theorem 6.1]{a-EK95}, \cite[Theorem 7.1.2]{phd-Kirillov95}. For any $\lambda\in P$ define 
  \begin{align}\label{eq:grad}
    \grad(e^\lambda)=e^\lambda\ot \partial_\lambda\in
    \overline{\C}[P]\ot \hfrak
  \end{align}
  and extend this definition to a map $\grad:
  \overline{\C}[P]\rightarrow \overline{\C}[P]\ot \hfrak$ by
  linearity and continuity. For any $f,g\in \overline{\C}[P]$ one has
  \begin{align}
    \grad(fg)&=f\grad(g)+g\grad(f), \label{eq:Gradfg}\\
    \Delta\lact (fg)&=(\Delta \lact f)g + 2 \grad(f)\lact g
    +f(\Delta\lact g).\label{eq:DeltaGrad}
  \end{align}
  Indeed, the above relations holds for $f=e^\lambda$, $g=e^\mu$ by
  \eqref{eq:Laplace} and \eqref{eq:grad}, and the general case
  follows from linearity and continuity. Relation \eqref{eq:DeltaGrad}
  implies that
  \begin{align}\label{eq:delDeldel}
    \deltah^{-1}\circ \Delta \circ \deltah = \Delta +2 \deltah^{-1}
    \grad(\deltah) +\deltah^{-1} (\Delta\lact\deltah).
  \end{align}
  From \eqref{eq:delta2delta} and \eqref{eq:Gradfg} one obtains
  \begin{align}\label{eq:delgrdel}
    \deltah^{-1} \grad(\deltah)= \frac{1}{2}\deltah_2^{-1} \grad{\deltah_2} = \partial_\rho + \sum_{\alpha\in \Phi^+}
    \frac{e^{-2\alpha}}{1-e^{-2\alpha}}\, \partial_\alpha.  
  \end{align}
  Comparing \eqref{eq:delDeldel} and \eqref{eq:delgrdel} with
  \eqref{eq:radial-part2} one obtains
  \begin{align}\label{eq:first-step}
    \deltah^{-1}\circ \Delta \circ \deltah  = \Pi_{0,0}(\Omega) +
    \deltah^{-1}(\Delta\lact \deltah)
  \end{align}
  and it remains to determine the function $\deltah^{-1}(\Delta\lact
  \deltah)$.  To this end observe that the denominator identity \cite[(10.4.4)]{b-Kac1} for
  $\deltah_2$ implies 
  \begin{align}\label{eq:denom-ident} 
    \Delta\lact \deltah_2=4(\rho,\rho)\deltah_2.
  \end{align}
  On the other hand \eqref{eq:DeltaGrad} and \eqref{eq:delta2delta} imply that
  \begin{align}\label{eq:Delta}
    \Delta\lact\deltah^2=2\deltah (\Delta \lact\deltah) +2 \grad(\deltah)\lact \deltah.
  \end{align}
  Combining \eqref{eq:denom-ident} and \eqref{eq:Delta} one obtains
  \begin{align*}
     \deltah^{-1}(\Delta\lact\deltah)=2(\rho,\rho)-
    \deltah^{-2}\grad(\deltah)\lact \deltah. 
  \end{align*}
  Observe further that for any $f,g\in \overline{\C}[P]$ one has
  $\grad(f)\lact g=(\grad(f), \grad (g))$ where the inner product is
  taken with respect to the second component of $\overline{\C}[P]\ot \hfrak$.
  Hence on obtains
  \begin{align}\label{eq:hilf}
    \deltah^{-1}(\Delta\lact\deltah)=2(\rho,\rho)-
    (\delta^{-1}\grad(\deltah) ,\delta^{-1}\grad(\deltah)). 
  \end{align}
  Now recall the definition of $\overline{\cR}_2[P]\ot \hfrak$ from
  subsection \ref{sec:frac-completion}, Remark \ref{rems:RP}.6 and define 
  \begin{align}\label{eq:v-def}
    v=\deltah^{-1} \grad(\deltah) = \partial_\rho + \sum_{\alpha\in \Phi^+} 
  \frac{e^{-2\alpha}}{1-e^{-2\alpha}}\, \partial_\alpha
  \end{align}
   as an element of $\overline{\cR}_2[P]\ot \hfrak$. 
  In view of \eqref{eq:first-step} and \eqref{eq:hilf} one needs to show that
   \begin{align}\label{eq:(v,v)}
    (v,v)=(\rho, \rho) + \sum_{\alpha\in \Phi^+}
    \frac{e^{-2\alpha}}{(1-e^{-2\alpha})^2} (\alpha,\alpha).  
  \end{align}
  in order to prove the proposition. To this end observe that
  \begin{align*}
    (v,v)=(\rho,\rho) + 2\sum_{\alpha\in
      \Phi^+}\frac{e^{-2\alpha}}{1-e^{-2\alpha}}(\alpha,\rho) +
    \sum_{\alpha,\beta\in \Phi^+}
    \frac{e^{-2(\alpha+\beta)}}{(1-e^{-2\alpha})(1-e^{-2\beta})} (\alpha,\beta).
  \end{align*}
  Define 
  \begin{align}\label{eq:X-def}
    X=(v,v)-(\rho,\rho) -\sum_{\alpha\in
      \Phi^+}\frac{e^{-2\alpha}}{(1-e^{-2\alpha})^2}(\alpha,\alpha). 
  \end{align}
  As $v$ is $W$-invariant by \eqref{eq:v-def} so is $(v,v)$ as an
  element of $\overline{\cR}_2[P]$. One verifies that the last summand
  of $X$ is also $W$-invariant in $\overline{\cR}_2[P]$. One calculates
   \begin{align}\label{eq:X-poles}
     X= \sum_{\alpha\in
      \Phi^+}\frac{e^{-2\alpha}}{1-e^{-2\alpha}}(\alpha,2\rho-\alpha) +
    \sum_{\alpha,\beta\in \Phi^+, \alpha\neq \beta}
    \frac{e^{-2(\alpha+\beta)}}{(1-e^{-2\alpha})(1-e^{-2\beta})}
    (\alpha,\beta). 
  \end{align}
  Hence $X$ has poles of order at most 1, and by Lemma \ref{lem:inA2} one obtains $\deltah_2 X\in A$.
  Moreover, $\deltah_2 X$ is $W$-antiivariant. Hence, by Remark \ref{rem:deltah2}, one gets $X\in A$. Moreover, $X\in
  \overline{\C}[P_0]$ and hence $X\in A_0^W$. By Lemma \ref{lem:AW0}
  one obtains  
  \begin{align*}
     X=\sum_{n\le n_0}a_n e^{n\delta}
  \end{align*}
  for some $a_n\in \C$, $n_0\in \N$. Collecting terms of the form
  $e^{-n\delta}$ in the expansion \eqref{eq:X-poles} one obtains
  \begin{align*}
    X= 2((\delta,\rho)-(\alpha_1,\alpha_1))\sum_{n=1}^\infty \frac{n
      e^{-2n\delta}}{1-e^{-2n\delta}}=0 
  \end{align*}
  which implies \eqref{eq:(v,v)} and hence concludes the proof of the proposition. 
\end{proof}

\subsection{Radial part of the Casimir for non-trivial one-dimensional representations}\label{sec:Cas-non-triv}
Let now $\eta$ and $\chi$ be non-trivial one-dimensional representations of
$\slzhth$. As a first step towards the description of
$\Pi_{\eta,\chi}(\Omega)$ we determine $\chi(X)$ for $X=Y+\thetah(Y)$
where $Y\in (\slzh)_\alpha$ is a root vector.
\begin{lem}\label{lem:chi-values}
  Let $a_0, a_1\in \C$ and $\chi=\chi_{a_0,a_1}$. Then
  \begin{align}
     \chi(f^{(n)}-e^{(-n)})&=\begin{cases}  
                             -i a_0 & \mbox{if $n\in \Z$ odd,}\\
                              i  a_1 & \mbox{if $n\in \Z$ even,} 
                           \end{cases}\label{eq:chi-fe}\\
     \chi(h^{(n)}-h^{(-n)})&=0 \qquad \mbox{for all $n\in \N_0$.}\label{eq:chi-h}
  \end{align}
\end{lem}
\begin{proof}
  The one-dimensional representation $\chi$ vanishes on commutators in $\slzhth$. Hence
  Equation \eqref{eq:chi-h} follows from
  $[e^{(0)}-f^{(0)},e^{(n)}-f^{(-n)}]=h^{(n)}-h^{(-n)}$. 
  Similarly, the relation
   \begin{align*}
     [h^{(1)}-h^{(-1)},e^{(n)}-f^{(-n)}]=2(e^{(n+1)}-f^{(-n-1)}) - 2(e^{(n-1)}-f^{(-n+1)})
  \end{align*}
  together with $\chi(f^{(0)}-e^{(0)})=\chi(iB_1)=ia_1$ and
  $\chi(f^{(1)}-e^{(-1)})=\chi(-iB_0)=-ia_0$ imply Equation \eqref{eq:chi-fe}. 
\end{proof}
Assume now that $a_0, a_1, b_0, b_1\in \C$ and that
$\chi=\chi_{a_0,a_1}$ and $\eta=\chi_{b_0,b_1}$. To distinguish between
the two cases in \eqref{eq:chi-fe} we introduce for any $n\in \Z$ the notation 
\begin{align*}
 \on=\begin{cases}1 & \mbox{if $n$ even,}\\ 0 & \mbox{if $n$
     odd.}  \end{cases}
\end{align*}
Using Proposition \ref{prop:radYtY} and Lemma
\ref{lem:chi-values} one obtains 
\begin{align*}
  \Pi_{\eta,\chi}(e^{(-n)}f^{(n)})&=\Pi_{0,0}(e^{(-n)}f^{(n)}) +
  \frac{e^{-n\delta+\alpha_1}(1+e^{-2n\delta+2\alpha_1})}{(1-
    e^{-2n\delta+2\alpha_1})^2}a_{\on} b_{\on}\\ & \qquad   - \frac{e^{-2n\delta+2\alpha_1}}{(1-
    e^{-2n\delta+2\alpha_1})^2}(a_{\on}^2+ b_{\on}^2), \\ 
  \Pi_{\eta,\chi}(h^{(-n)}h^{(n)})&=\Pi_{0,0}(h^{(-n)}h^{(n)}),\\
  \Pi_{\eta,\chi}(f^{(-n)}e^{(n)})&=\Pi_{0,0}(f^{(-n)}e^{(n)}) +
  \frac{e^{-n\delta-\alpha_1}(1+e^{-2n\delta-2\alpha_1})}{(1-
    e^{-2n\delta-2\alpha_1})^2}a_{\on} b_{\on}\\ & \qquad   - \frac{e^{-2n\delta-2\alpha_1}}{(1-
    e^{-2n\delta-2\alpha_1})^2}(a_{\on}^2+ b_{\on}^2), \\ 
\end{align*}
Hence Equations \eqref{eq:Casimir}, \eqref{eq:PiOmega-fin}, and
\eqref{eq:sum-uu} yield
\begin{align*}
 & \Pi_{\eta,\chi}(\Omega)=  \Pi_{0,0}(\Omega)-
  2\frac{e^{2\alpha_1}}{(1-e^{2\alpha_1})^2}(a_1^2+b_1^2)+ 2
  \frac{e^{\alpha_1}(1+e^{2\alpha_1})}{(1-e^{2\alpha_1})^2}a_1b_1 \\ 
  &-2\sum_{n=0}^\infty\Bigg(\Big(
  \frac{e^{-2(2n+1)\delta+2\alpha_1}}{(1-e^{-2(2n+1)\delta+2\alpha_1})^2}
  + \frac{e^{-2(2n+1)\delta-2\alpha_1}}{(1-e^{-2(2n+1)\delta-2\alpha_1})^2}\Big)
  (a_0^2+b_0^2)\\
  &\qquad- 
  \frac{e^{-(2n+1)\delta+\alpha_1}(1+e^{-2(2n+1)\delta+2\alpha_1})}{(1-
    e^{-2(2n+1)\delta+2\alpha_1})^2}a_0b_0 
  -\frac{e^{-(2n+1)\delta-\alpha_1}(1+e^{-2(2n+1)\delta-2\alpha_1})}{(1  
    -e^{-2(2n+1)\delta-2\alpha_1})^2}a_0b_0 \Bigg) \\
  &- 2\sum_{n=1}^\infty\Bigg(\Big(
  \frac{e^{-2(2n)\delta+2\alpha_1}}{(1-e^{-2(2n)\delta+2\alpha_1})^2}
  + \frac{e^{-2(2n)\delta-2\alpha_1}}{(1-e^{-2(2n)\delta-2\alpha_1})^2}\Big)
  (a_1^2+b_1^2)\\ &\qquad -
  \frac{e^{-(2n)\delta+\alpha_1}(1+e^{-2(2n)\delta+2\alpha_1})}{(1-
    e^{-2(2n)\delta+2\alpha_1})^2}a_1b_1 -
  \frac{e^{-(2n)\delta-\alpha_1}(1+ e^{-2(2n)\delta-2\alpha_1})}{(1  
    -e^{-2(2n)\delta-2\alpha_1})^2}a_1b_1 \Bigg). 
\end{align*}
To shorten formulas we abbreviate $q=e^{-\delta}$ and $y=e^{\alpha_1}$ and obtain
\begin{align*}
  \Pi_{\eta,\chi}(\Omega)&= 
  \Pi_{0,0}(\Omega)- 2\frac{y^{2}}{(1-y^{2})^2}(a_1\pm b_1)^2 + 2
  \frac{y}{(1 \mp y)^2}a_1b_1 \\
  &-2\sum_{n=0}^\infty\Bigg(\Big(
  \frac{q^{2(2n+1)}y^2}{(1-q^{2(2n+1)}y^2)^2}
  + \frac{q^{2(2n+1)}y^{-2}}{(1-q^{2(2n+1)}y^{-2})^2}\Big)
  (a_0 \pm b_0)^2\\
   &\qquad- 
  \frac{q^{(2n+1)}y}{(1 \mp
    q^{(2n+1)}y)^2}a_0b_0 
  -\frac{q^{(2n+1)}y^{-1}}{(1 \mp q^{(2n+1)}y^{-1})^2}a_0b_0 \Bigg) \\
 &-2\sum_{n=1}^\infty\Bigg(\Big(
  \frac{q^{2(2n)}y^2}{(1-q^{2(2n)}y^2)^2}
  + \frac{q^{2(2n)}y^{-2}}{(1-q^{2(2n)}y^{-2})^2}\Big)
  (a_1 \pm b_1)^2\\
   &\qquad- 
  \frac{q^{(2n)}y}{(1 \mp
    q^{(2n)}y)^2}a_1b_1 
  -\frac{q^{(2n)}y^{-1}}{(1 \mp q^{(2n)}y^{-1})^2}a_1b_1 \Bigg).
\end{align*}
Hence formulas \eqref{eq:p01}, \eqref{eq:p11}, and the convention \eqref{eq:wp-formal} for the Weierstra{\ss} $\wp$-functions from Appendix \ref{app:theta} yield
\begin{align*}
  \Pi_{\eta,\chi}(\Omega)=& 
  \Pi_{0,0}(\Omega) + \frac{1}{2\pi^2} \Big( -\wp_{0,1}(y^2,q^2)(a_0+b_0)^2 + \wp_{0,1}(y,q)a_0b_0 \\
  &\qquad \qquad \qquad \qquad - \wp_{1,1}(y^2,q^2)(a_1+b_1)^2 + \wp_{1,1}(y,q)a_1b_1\Big)
\end{align*}
With \eqref{eq:wp01x22} and \eqref{eq:wp11x22} this becomes
\begin{align}  
  \Pi_{\eta,\chi}(\Omega) 
  =& \Pi_{0,0}(\Omega) - \frac{1}{2\pi^2} \Big( \frac{1}{4} (\wp_{0,1}(y,q){+}\wp_{0,0}(y,q))
  (a_0+b_0)^2 -\wp_{0,1}(y,q)a_0b_0  \nonumber \\
  & \qquad \qquad \qquad \quad +\frac{1}{4}  (\wp_{1,1}(y,q){+}\wp_{1,0}(y,q))(a_1+b_1)^2 -
  \wp_{1,1}(y,q)a_1b_1 \Big) \nonumber\\ 
=& \Pi_{0,0}(\Omega) -\frac{1}{8\pi^2} \Big( \wp_{0,1}(y,q) (a_0-b_0)^2 + \wp_{0,0}(y,q)(a_0+b_0)^2\label{eq:rad-etachi}\\
   & \qquad \qquad \qquad \qquad + \wp_{1,1}(y,q)(a_1-b_1)^2 + \wp_{1,0}(y,q)(a_1+b_1)^2 \Big).\nonumber
\end{align}
Conjugating by the element $\deltah\in \overline{\C}[P]$ and inserting
the expression for $\deltah \Pi_{0,0}(\Omega) \deltah^{-1}$
from Proposition \ref{prop:Om-conj} one obtains 
\begin{align*}
 \deltah  \Pi_{\eta,\chi}&(\Omega) \deltah^{-1}=\Delta-
 (\rho,\rho) +\frac{1}{2\pi^2} \wp_{1,1}(y^2,q) -\frac{1}{8\pi^2} \Big(\wp_{0,1}(y,q) (a_0-b_0)^2 \\
   & \qquad + \wp_{0,0}(y,q)(a_0+b_0)^2  +\wp_{1,1}(y,q)(a_1-b_1)^2 + \wp_{1,0}(y,q)(a_1+b_1)^2 \Big)
\end{align*}
Applying \eqref{eq:wp01x22} -- \eqref{eq:wp11x21} one obtains the first part of the following result. The second part is a consequence of Theorem \ref{thm:rad-part}.
\begin{thm}\label{thm:del-Om}
Let $\chi=\chi_{a_0,a_1}$ and $\eta=\chi_{b_0,b_1}$. Then one has
\begin{align*}
  \deltah  \Pi_{\eta,\chi}(\Omega) \deltah^{-1}=&\Delta- (\rho,\rho) - \frac{1}{8\pi^2}\Big( \wp_{0,1}(y,q) ((a_0{-}b_0)^2{-}1) +
   \wp_{0,0}(y,q)((a_0{+}b_0)^2{-}1) \\
  & \qquad \qquad \quad\quad+ \wp_{1,1}(y,q)((a_1{-}b_1)^2{-}1) +
    \wp_{1,0}(y,q)((a_1{+}b_1)^2{-}1)\Big)
\end{align*}
as an element of $\overline{\C}[P]\rtimes U(\hfrak)$, where $y=e^{\alpha_1}$ and $q=e^{-\delta}$. If, moreover, $\lambda\in P^+_K$ and $\varphi\in \overline{\delta V(\lambda)}^\chi \ot \overline{V(\lambda)}^\eta \subset \overline{\C}[P_K]$ is a zonal spherical function then $\deltah \Psi(\varphi) \in \overline{\C}[P_K]$ satisfies the eigenvalue equation
\begin{align*}
  \deltah  \Pi_{\eta,\chi}(\Omega) \deltah^{-1} \lact (\deltah \Psi(\varphi))=(\lambda,\lambda+\rho) \deltah \Psi(\varphi).
\end{align*}
\end{thm}
\section{Functional interpretation}\label{sec:functional}
So far we have considered the radial part $\Pi_{\eta,\chi}(\Omega)$ only as and element of $\overline{\C}[P]\rtimes U(\hfrak)$ or as an operator on $\overline{\C}[P]$. Similarly, twisted zonal spherical functions were considered as elements of $A$ or $\overline{\C}[P]$. In Subsection \ref{sec:wp-function}, following \cite{a-EK95}, \cite{phd-Kirillov95}, we introduce a notion of convergence for elements of $\overline{\C}[P]$ which provides holomorphic functions.   The aim of Subsection \ref{sec:Ino} is to interpret $\deltah\Pi_{\eta,\chi}(\Omega)\deltah^{-1}$ as a differential operator on a domain $D$. In Subsection \ref{sec:convergence} we prove, for $\chi=\eta$, that twisted zonal spherical functions converge on a suitable domain inside $D$.
\subsection{Functional interpretation of elements of $\overline{\C}[P]$}\label{sec:wp-function}
Recall from \eqref{eq:f(nu)} that elements of $\C[P]$ may be considered as functions on $\hfrak$. Choose coordinates on $\hfrak$ via the map
\begin{align*}
  \C^3 \rightarrow \hfrak, \qquad (z,u,\tau) \mapsto \pi i (z h_1 + u c - \tau d).
\end{align*}
Then $e^\lambda\in \C[P]$ for $\lambda=c_1 \alpha_1 + K \varpi_0 + n \delta$ turns into the function on $\C^3$ given by 
\begin{align}\label{eq:elambda}
 e^{\lambda}(z,u,\tau) = e^{\pi i (2 c_1 z +K u - n\tau)}.
\end{align}
One obtains in particular
  $e^{\alpha_1}(z,u,\tau)=e^{2 \pi i z}$ and $e^{-\delta}(z,u,\tau) = e^{\pi i \tau}$.
Let $\cH=\{\tau\in \C\,|\,\mathrm{Im}(\tau)>0\}$ denote the upper half plane and 
\begin{align*}
  Y=\C\times\C\times \cH.
\end{align*} 
We consider elements $\sum_{\lambda\in P} a_\lambda e^\lambda\in \overline{\C}[P]$ as functions on $Y$ if the series of functions converges. More precisely, we use the following definition from \cite[Definition 8.1.1]{phd-Kirillov95}. Recall Remark \ref{rems:RP}.6 and let $\overline{\tau}_1:\overline{\cR}_2[P]\rightarrow \overline{\C}[P]$ denote the map extending \eqref{eq:tau-def} which expands $(1-e^{2\alpha})^{-1}$ into the geometric series for any $\alpha\in \Phi^-$.
\begin{defi}\label{defi:convergence}
  Let $f=\sum_{\lambda\in P}  a_\lambda e^\lambda\in \overline{\C}[P]$ and let $U\subseteq Y$ be an open subset. We say that $f$ converges on $U$ if this series, considered as a series of functions on $Y$ converges absolutely and locally uniformly in $U$. Similarly, if $f\in \overline{\cR}_2[P]$ then we say that $f$ converges on $U$ if $\overline{\tau}_1(f)\in \overline{\C}[P]$ converges on $U$.
\end{defi}
Recall from Appendix \ref{app:theta} that the Weierstrass $\wp$-functions $\wp_{i,j}(z,\tau)$ converge on the domain
\begin{align*}
  D=\{(z,\tau)\in \C\times \cH\,|\, -\mathrm{Im}(\tau)/2 < \mathrm{Im}(z)<0\}.
\end{align*}
On this domain the $\wp$-functions are obtained from the formal $\wp$-functions $\wp_{i,j}(e^{\alpha_1},e^{-\delta})$ in the sense of Definition \ref{defi:convergence}. We sketch a proof to indicate that this fact may not be quite as obvious as it appears. By slight abuse of notation we also use the symbol $D$ to denote the subset of all $(z,u,\tau)\in Y$ for which $-\mathrm{Im}(\tau)/2<\mathrm{Im}(z)<0$.
\begin{prop}\label{prop:wp-convergence}
For any $i,j\in \{0,1\}$ the element $\wp_{i,j}(e^{\alpha_1},e^{-\delta})\in \overline{\cR}_2[P]$ converges on $D$. The holomorphic function which is obtained as the limit satisfies  $\wp_{i,j}(e^{\alpha_1},e^{-\delta})(z,u,\tau)=\wp_{ij}(z,\tau)$.
\end{prop}
\begin{proof}[Sketch of proof]
  We only consider $\wp_{1,0}$, see \eqref{eq:p10}, the other cases are treated analogously. Moreover, we suppress the variable $u$ because all occurring functions are independent of $u$. For $\beta=2 n\delta \pm \alpha_1 \in \Phi^+$ write
  \begin{align}\label{eq:wp-frac}
    \overline{\tau}_1\left( \frac{e^{-\beta}}{(1+e^{-\beta})^2} \right)=\overline{\tau}_1\left( \frac{e^{-\beta}(1-e^{-\beta})^2}{(1-e^{-2\beta})^2} \right) = \sum_{\lambda\in P} a_{\beta,\lambda} e^\lambda \in \overline{\C}[P]
  \end{align}
for some coefficients $a_{\beta, \lambda}\in \C$. For $(z,\tau)\in D$ one has
\begin{align*}
  |e^{-\beta}(z,u,\tau)|= |e^{i\pi (\mp 2z + n\tau)}|= e^{\pi(\pm 2 \mathrm{Im}(z)-n\mathrm{Im}(\tau))}<1.
\end{align*}
 and hence the above series \eqref{eq:wp-frac} converges on $D$. By definition one has
\begin{align}\label{eq:tau-wp10}
  \overline{\tau}_1\left(\wp_{1,0}(e^{\alpha_1},e^{-\delta})\right) = -4 \pi^2\sum_{\lambda\in P}\left(\sum_{\beta=2n\delta\pm \alpha_1\in \Phi^+} a_{\beta,\lambda}\right)e^\lambda.
\end{align}  
Note the order of summation. For any $x=(z,\tau)\in D$ there exists $x^-=(1/2+i \mathrm{Im}(z),i\mathrm{Im}(\tau))\in D$ such that 
\begin{align*}
  -|e^{-\beta}(x)|&=e^{-\beta}(x^-)
\end{align*} 
independently of $\beta$. Hence, using the explicit form of the coefficients in the expansion \eqref{eq:wp-frac}, one obtains
   \begin{align}\label{eq:abs-conv-}
  -\sum_{\lambda\in P} |a_{\beta,\lambda} e^\lambda(x)| = \sum_{\lambda\in P} a_{\beta,\lambda} e^\lambda (x^-) = \frac{e^{-\beta}(x^-)}{(1+e^{-\beta}(x^-))^2}.
\end{align}
The convergence of $\wp_{0,1}(1/2 + i\mathrm{Im}(z),i\mathrm{Im}(\tau))$ and \eqref{eq:abs-conv-} imply that the double sum
\begin{align*}
  \sum_{\beta=2n\delta\pm \alpha_1\in \Phi^+} \,\sum_{\lambda\in P} |a_{\beta,\lambda}e^\lambda(x)|
\end{align*}
converges. Hence the order of summation may be exchanged and \eqref{eq:tau-wp10} also converges absolutely in $x$. The locally uniform convergence of \eqref{eq:tau-wp10} is obtained using the locally uniform convergence of $\wp_{1,0}(z,\tau)$ and of \eqref{eq:wp-frac}, the absolute convergence shown above, and an $\vep/3$-argument.
\end{proof}
The following are well-known examples of elements in $\overline{\C}[P]$ which converge on $Y$ in the sense of Definition \ref{defi:convergence}, see \cite[10.6]{b-Kac1}, \cite[Theorem 8.1.2]{phd-Kirillov95}.
\begin{eg}\label{eg:mmu}
  For any $\lambda \in P^+$ the orbit sum $m_\lambda=\sum_{\mu\in W\lambda}e^\mu$ converges on $Y$.
\end{eg}
\begin{eg}\label{eg:deltah1}
 The affine Weyl denominator $\deltah_1\in\overline{\C}[P]$ defined in \eqref{eq:deltah12} converges on Y. 
\end{eg}
Example \ref{eg:deltah1} implies that $\deltah_2=e^{2\rho}\prod_{\beta\in \Phi^+}(1-e^{-2\beta})\in \overline{\C}[P]$ also converges on $Y$. 
Recall the element $\deltah\in \overline{\C}[P]$ with $\deltah^2=\deltah_2$ defined in \eqref{eq:deltah-def}. The exponent
  \begin{align*}
     \sum_{\alpha\in \Phi^+} \sum_{n\in \N} \frac{e^{-2n\alpha}}{n}
  \end{align*}
converges on $D$ and hence so does $\deltah$. This also provides a direct proof of the convergence of $\deltah_2$ on $D$. 
\subsection{The Heun-KZB heat operator}\label{sec:Ino}
With these preparations we can turn to the functional interpretation of the radial part $\deltah \Pi_{\eta,\chi}(\Omega) \deltah^{-1}$ from Theorem \ref{thm:del-Om}. The operator $\Delta=2cd+h_1^2/2$ satisfies
\begin{align*}
  \Delta \lact e^\lambda = (2 K n + 2 c_1^2) e^\lambda \qquad \mbox{if $\lambda=c_1 \alpha_1 + K \varpi_0 + n \delta$}
\end{align*}
and hence
\begin{align*}
  \Delta \lact e^\lambda (z,u,\tau) &= (2 K n + 2 c_1^2) e^{\pi i (2 c_1 z +K u - n\tau)}\\
  &= \frac{1}{2\pi^2}
  \big(4 \pi i K \frac{\partial}{\partial \tau} - \frac{\partial^2}{\partial z^2}\big) e^\lambda(z,u,\tau).
\end{align*}
By Theorem \ref{thm:del-Om} and Proposition \ref{prop:wp-convergence} the conjugated radial part $\deltah \Pi_{\eta,\chi}(\Omega) \deltah^{-1}$ hence corresponds to the operator
\begin{align*}
   &\frac{1}{2\pi^2}
  \big(4 \pi i K \frac{\partial}{\partial \tau} - \frac{\partial^2}{\partial z^2}\big)
 -\frac{1}{2} -\frac{1}{8\pi^2}\Big( \wp_{0,1}(z,\tau) ((a_0{-}b_0)^2{-}1)\\
  &+ \wp_{0,0}(z,\tau)((a_0{+}b_0)^2{-}1) + \wp_{1,1}(z,\tau)((a_1{-}b_1)^2{-}1) +
    \wp_{1,0}(z,\tau)((a_1{+}b_1)^2{-}1)\Big).
\end{align*}
on $D$ where we used that $(\rho,\rho)=1/2$. By \eqref{eq:P00} -- \eqref{eq:P10} the functions $\wp_{i,j}$ may be identified with the shifted Weierstrass functions $\wp_k$, for $k=0, 1, 2, 3$, which are defined by
\begin{align*}
  \wp_0(z,\tau)&=\wp(z,\tau), &  \wp_1(z,\tau)&=\wp(z+1/2,\tau),\\
  \wp_2(z,\tau)&=\wp(z+(1+\tau)/2,\tau), &  \wp_3(z,\tau)&=\wp(z+\tau/2,\tau).
\end{align*}
Recall $\eta_1(\tau)$ from Equation \eqref{eq:theta11-p} which appears as a summand in this identification. With the notation
\begin{equation}\label{eq:ell-def}
  \begin{aligned}
  l_0&=(a_1-b_1-1)/2, & l_1&=(a_1+b_1-1)/2, \\  l_2&=(a_0+b_0-1)/2, & l_3&=(a_0-b_0-1)/2
  \end{aligned}
\end{equation}  
we obtain the following result.
\begin{thm}\label{thm:HeunKZBheat}
  Let $\chi=\chi_{a_0,a_1}$ and $\eta=\chi_{b_0,b_1}$, let $\ell=(l_0,l_1,l_2,l_3)$ be given by \eqref{eq:ell-def}, and let $f\in \overline{\C}[P_K]$ be convergent on some open subset $D'\subseteq D$. Then the operator $\deltah \Pi_{\eta,\chi}(\Omega) \deltah^{-1}$ acts on the holomorphic function $f|_{D'}$ as
\begin{align*}
  H_{K,\ell}=\frac{1}{2\pi^2}\Big( 4  \pi i K & \frac{\partial}{\partial \tau} - \frac{\partial^2}{\partial z^2} - \pi^2 + \sum_{k=0}^3 l_k(l_k+1) \wp_k(z,\tau) + c(\tau)\Big)
\end{align*} 
where $c(\tau)=-\eta_1(\tau)\sum_{k=0}^3 l_k(l_k+1)$.
\end{thm}
We call $H_{K,\ell}$ the Heun-KZB heat operator of level $K$.
Up to the overall factor, the constant $\pi^2$ and the summand $c(\tau)$ which can be integrated, the operator $H_{K,\ell}$ is the desired Heun-version of the KZB-heat operator which appears for example in Varchenko's book \cite[p.~70]{b-Varchenko03}.
\subsection{Convergence of eigenfunctions for $\chi=\eta$}\label{sec:convergence}
Let 
\begin{align*}
  \varphi\in \overline{\delta V(\lambda)}^\chi\ot\overline{V(\lambda)}^\eta 
\end{align*}
be a twisted zonal spherical function corresponding to a highest weight $\lambda$ of level $K$. By Theorem \ref{thm:rad-part} $\Psi(\varphi)\in A$ satisfies the relation
\begin{align}\label{eq:eigeneq}
 \Pi_{\eta,\chi}(\Omega)\lact \Psi(\varphi) = (\lambda, \lambda+\rho) \Psi(\varphi).
\end{align}
We would like to show that the eigenfunction $\Psi(\varphi)$ converges on $D$ and hence provides a holomorphic eigenfunction to the Heun-KZB heat operator $H_{K,\ell}$. The analog result in Etingof and Kirillov's paper \cite[Theorem 8.2]{a-EK95}, \cite[Theorem 8.1.5]{phd-Kirillov95} relies on the $W$-invariance of $\Psi(\varphi)$. For this reason we restrict to the case $\eta=\chi=\chi_{a_0,a_1}$, see Proposition \ref{prop:invariant}. 

Introduce the notation
  $M_\chi=\Pi_{\chi,\chi}(\Omega)$
and observe that by \eqref{eq:rad-etachi} one has $M_\chi = M_0 + \cS$ where 
\begin{align*}
  \cS=- \frac{1}{2 \pi^2}\Big( a_0^2 \wp_{0,0}(y,q) + a_1^2 \wp_{1,0}(y,q) \Big)\in \overline{\cR}_2[P].
\end{align*}
In general, $M_\chi$ does not map $A^W$ to itself. For example, for the constant function $\mathbf{1}$ one obtains $M_\chi \mathbf{1} = \cS \notin A$. To circumvent this problem in the following proof we will multiply by $\deltah^4=\deltah_2^2$. This, however, only allows us to show that $\Psi(\varphi)$ converges if $\mathrm{Im}(\tau)$ is large. 
\begin{prop}
  Let $\lambda \in \oP^+_K(\chi,\chi)$ and $\varphi\in \overline{\delta V(\lambda)}^\chi \ot \overline{V(\lambda)}^\chi$. Then there exists $T\ge 0$ such that $\Psi(\varphi)$ converges on the region $\{(h,u,\tau)\in D\,|\, \mathrm{Im}(\tau)>T\}$.
\end{prop}
\begin{proof}
  By Proposition \ref{prop:invariant} we have $\Psi(\varphi)\in A^W_K$ and hence, by Proposition \ref{prop:AWKbase}, we can write 
  \begin{align*}
    \Psi(\varphi) = \sum_{\mu \in \oP_K^+} f_\mu m_\mu
  \end{align*}
  where $f_\mu\in \C((q))$ are formal Laurent series in $q=e^{-\delta}$. By Example \ref{eg:mmu} the $m_\mu$ converge on $Y$. Hence it remains to show that $f_\mu(e^{i\pi\tau})$ converges if $\mathrm{Im}(\tau)$ is large. Equivalently we show that there exists $0<\vep<1$ such that the Laurent series $f_\mu(q)$ converges if $0<|q|<\vep$.
  
By \cite[Theorem 6.1.(3)]{a-EK95}, \cite[Theorem 7.1.4]{phd-Kirillov95} we know that $M_0 A^W \subseteq A^W$. Hence there exist $a_{\mu,\nu}\in \C((q))$ such that 
\begin{align*}
  M_0 m_\mu = \sum_{\nu \in \oP_K^+} a_{\mu,\nu} m_\nu.
\end{align*}  
As in the proof of \cite[Theorem 8.1.5]{phd-Kirillov95} this implies that $a_{\mu,\nu}(q)$ converges if $0<|q|<1$. Equation \eqref{eq:eigeneq} can be rewritten as
\begin{align}
  \sum_{\mu,\nu\in \oP_K^+} f_{\mu} a_{\mu,\nu} m_\nu  - &2(K+2)\sum_{\mu\in \oP_K^+} \left(q\frac{\partial}{\partial q} f_{\mu}\right) m_\mu + \sum_{\mu\in \oP_K^+} f_{\mu} \cS m_\mu\nonumber\\
    &= (\lambda,\lambda+\rho) \sum_{\mu\in \oP_K^+} f_{\mu} m_\mu.\label{eq:rewrite}
\end{align}
We want to compare coefficients of the $m_\mu$, however, to this end we need to manipulate $\cS m_\mu$ in the third summand. 

By Lemma \ref{lem:inA2}, using that $\cS\in \overline{\cR}_2[P]^W$, we get $\deltah_2^2 \cS\in A^W$. Moreover, $\deltah_2$ and $\cS$ define analytic functions for $\mathrm{Im}(\tau)> 2 |\mathrm{Im}(z)|$. Hence
\begin{align*}
  \deltah_2^2 \cS m_\mu &=\sum_{\nu \in \oP_{K+4}^+} g_{\mu,\nu} m_\nu\\
  \deltah_2^2  m_\mu    &=\sum_{\nu \in \oP_{K+4}^+} b_{\mu,\nu} m_\nu
\end{align*}  
for some $g_{\mu,\nu}, b_{\mu,\nu}\in \C((q))$ which define analytic functions on $Y$. Hence, multiplying Equation \eqref{eq:rewrite} by $\deltah_2^2$ and comparing coefficients of $m_\nu$ for $\nu\in \oP_{K+4}^+$, one obtains 
\begin{align}
  \sum_{\mu,\eta\in \oP_K^+} f_{\mu} a_{\mu,\eta} b_{\eta,\nu} - &2(K+2)\sum_{\mu\in \oP_K^+} \left(q\frac{\partial}{\partial q} f_{\mu}\right) b_{\mu,\nu} + \sum_{\mu\in \oP_K^+} f_{\mu} g_{\mu,\nu}  \nonumber\\
    &= (\lambda,\lambda+\rho) \sum_{\mu\in \oP_K^+} f_{\mu} b_{\mu,\nu}.\label{eq:rewrite2}
\end{align}
Observe that $B=(b_{\eta,\nu})_{\eta\in \oP^+_K, \nu\in \oP^+_{K+4}}$ is a $(K+1)\times (K+5)$ matrix of analytic functions. Multiplication by $\deltah_2^2$ defines an injective linear map $A^W_K\rightarrow A^W_{K+4}$ of $A_0^W$ vector spaces which is represented by right multiplication by $B$. We can order the set $\{m_\mu\,|\, \mu \in \oP_{K+4}^+\}$ in such a way that the first $(K+1)\times(K+1)$ block $B_{k+1}$ of $B$ has full rank. Hence there exists a right inverse $(K+5)\times(K+1)$ matrix $C=(c_{\nu,\kappa})_{\nu\in \oP^+_{K+4}, \kappa\in \oP^+_{K}}$ for which the four bottom rows are zero.  There exist $0<\vep<1$ such that
\begin{align*}
  \det(B_{k+1})(q)\neq 0 \qquad \qquad \mbox{for all $q\in \C$ with $0<|q|<\vep$}.  
\end{align*}
Hence, by Cramer's rule, the entries of the matrix $C$ are analytic in all points $q\in \C$ with $0<|q|<\vep$. Multiplying Equation \eqref{eq:rewrite2} by $c_{\nu,\eta}$ and summing over $\nu$ one obtains
\begin{align}
  \sum_{\mu,\eta\in \oP_K^+} f_{\mu} a_{\mu,\eta} - &2(K+2)\sum_{\mu\in \oP_K^+} \left(q\frac{\partial}{\partial q} f_{\eta}\right) + \sum_{\mu\in \oP_K^+, \nu\in \oP_{K+4}^+} f_{\mu} g_{\mu,\nu} c_{\nu,\eta}  \nonumber\\
    &= (\lambda,\lambda+\rho) \sum_{\mu\in \oP_K^+} f_{\eta}.\label{eq:rewrite3}
\end{align}
For $0<|q|<\vep$ this is a system of first order differential equations with analytic coefficients and hence $f_{\mu}$ are analytic functions in this domain. 
\end{proof}

\appendix

\section{Theta functions and the Weierstrass $\wp$-function}\label{app:theta}
In this appendix we recall well-known presentations of theta functions and the corresponding Weierstrass $\wp$-functions following Chapter I of \cite{b-Mum83}.
\subsection{Basics on $\theta(z,\tau)$} \label{sec:basictheta}
The theta function  
  \begin{align*}
    \theta(z,\tau)=\sum_{n\in \Z} e^{\pi i n^2 \tau + 2 \pi i n z}
  \end{align*}
is defined as a function on $\C\times \mathcal{H}$ where $\mathcal{H}=\{\tau\in \C\,|\,\mathrm{Im}(\tau)>0\}$. The series is normally convergent in the sense of, say, \cite[Definition III.1.4]{b-FreiBus09} and hence converges absolutely and uniformly on compact sets \cite[p.~1]{b-Mum83}. Hence $\theta$ defines a holomorphic function on $\C\times \mathcal{H}$. Moreover, $\theta$ is quasi periodic
\begin{align*}
  \theta(z+1,\tau)=\theta(z,\tau), \qquad \theta(z+\tau,\tau)=e^{-\pi i \tau - 2 \pi i z}\theta(z,\tau)
\end{align*} 
and it is the most general entire function with 2 quasi-periods. Associated to $\theta(z,\tau)$ there are four theta functions with characteristics
\begin{align}
    \theta_{0,0}(z,\tau)&=\theta(z,\tau) \label{eq:theta00}\\
    \theta_{0,1}(z,\tau)&= \theta(z+\tfrac{1}{2},\tau)\label{eq:theta01}\\
    \theta_{1,0}(z,\tau)&= e^{\pi i \tau/4 + \pi i z}\theta(z+\tfrac{\tau}{2}, \tau)\label{eq:theta10}\\
    \theta_{1,1}(z,\tau)&= e^{\pi i \tau/4 + \pi i (z + 1/2)}\theta(z+\tfrac{\tau}{2}+\tfrac{1}{2}, \tau).\label{eq:theta11}
  \end{align}
which can be used to obtain a projective embedding of the elliptic curve $E_\tau=\C/(\Z+\Z\tau)$ into 3-dimensional complex projective space, see \cite[\textsection\, I.4]{b-Mum83}. All of the above functions are holomorphic on $\C\times \mathcal{H}$. Moreover, $\theta(z,\tau)=0$ if and only if $z\equiv (1+\tau)/2$ mod $\Z+\Z\tau$ and $\theta$ has a simple zero in these points. Hence $\theta_{1,1}$ has only simple zeros which lie precisely in the points of the lattice $\Z+\Z\tau$.

By \cite[Proposition 14.1]{b-Mum83} we have the following product expansion
 \begin{align*}
   \theta(z,\tau)= \prod_{m\in \N}(1-e^{\pi i (2m) \tau}) \prod_{m\in \N_0} \big\{(1+e^{\pi i (2m+1)\tau- 2 \pi i z}) (1+e^{\pi i (2m+1)\tau+ 2 \pi i z}) \big\}
 \end{align*}  
With the abbreviations $q=e^{\pi i \tau}$, $y=e^{2 \pi i z}$, and $C(q)=\prod_{m\in \N}(1-q^{2m})$ the above formula becomes
\begin{align*}
  \theta(z,\tau)= C(q) \prod_{m\in \N_0} \big\{(1+q^{2m+1}y^{-1}) (1+q^{2m+1}y) \big\}.
\end{align*} 
In view of the relations \eqref{eq:theta00} -- \eqref{eq:theta11} one obtains for the functions $\theta_{ij}$, $i,j\in \{0,1\}$, the product expansions
\begin{align*}
\theta_{0,0}(z,\tau)&= C(q)\prod_{m\in \N_0} (1+q^{2m+1}y^{-1}) (1+q^{2m+1}y),\\
  \theta_{01}(z,\tau)&= C(q)\prod_{m\in \N_0} \big\{(1-q^{2m+1}y^{-1}) (1-q^{2m+1}y) \big\} ,\\
  \theta_{10}(z,\tau)&= e^{\pi i \tau/4 + \pi i z} (1+y^{-1}) C(q) \prod_{m\in \N} (1+q^{2m}y^{-1}) (1+q^{2m}y),\\
  \theta_{11}(z,\tau)&=e^{\pi i \tau/4 + \pi i (z+1/2)} (1-y^{-1}) C(q) \prod_{m\in \N} (1-q^{2m}y^{-1}) (1-q^{2m}y). 
\end{align*}
\subsection{From theta functions to Weierstrass $\wp$-functions}\label{ap:theta}
The Weierstrass $\wp$-function is defined by
\begin{align*}
  \wp(z,\tau)=\frac{1}{z^2}+\sum_{(m,n)\in \Z^2\setminus (0,0)}\left(\frac{1}{(z-m-n\tau)^2}- \frac{1}{(m+n\tau)^2}\right).
\end{align*}
For fixed $\tau$, the function $\wp$ is an elliptic function with poles of order 2 at all points of the lattice $\Z+\Z\tau$. The Weierstrass zeta function is the meromorphic function defined on $\C\times \cH$ by
\begin{align*}
  \zeta(z,\tau)=\frac{1}{z}+\sum_{(m,n)\in \Z^2\setminus (0,0)}\left(\frac{1}{z-m-n\tau}+ \frac{1}{m+n\tau} + \frac{z}{(m+n\tau)^2}\right).
\end{align*}
One has $\zeta'=-\wp$. As $\wp$ is elliptic, there exists for every $\gamma\in \Z+\Z\tau$ a constant $\eta_\gamma(\tau)\in \C$ such that
\begin{align}\label{eq:zeta-eta}
  \zeta(z+\gamma,\tau)=\zeta(z,\tau) + \eta_\gamma(\tau) \qquad \mbox{for all $z\in \C\setminus(\Z+\Z\tau)$}. 
\end{align}
For fixed $\tau$, as explained in Subsection \ref{sec:basictheta}, this implies that the function
\begin{align*}
  \frac{\theta_{1,1}'}{\theta_{1,1}}
\end{align*}
is meromorphic with simple poles at all points of $\Z+\Z\tau$. Its derivative hence has poles of order two at all points of the lattice. Moreover, the quasi-periodicity of $\theta(z,\tau)$ implies that $\left( \theta_{1,1}'/\theta_{1,1}\right)'$ is a doubly periodic function in $z$. Hence 
\begin{align*}
  \left( \theta_{1,1}'/\theta_{1,1}\right)'= A \wp + B
\end{align*}   
for some constants $A,B$ which only depend on $\tau$. More explicitly, one has by \cite[Theorem 3.9]{b-Polishchuk03} the relation
\begin{align}\label{eq:theta11-p}
  \left( \theta_{1,1}'/\theta_{1,1}\right)'(z,\tau)= - \wp(z,\tau) - \eta_1(\tau)
\end{align}   
where $\eta_1(\tau)$ is defined by Equation \eqref{eq:zeta-eta} above for $\gamma=1$.

We now define for $i,j\in \{0,1\}$ variants of the Weierstrass $\wp$-function by
\begin{align*}
  \wp_{i,j}(z,\tau)=\left( \frac{\theta_{i,j}'(z,\tau)}{\theta_{i,j}(z,\tau)} \right)'
\end{align*}
where as before prime $\, ' \,$ denotes differentiation by $z$.
As the product expansion for $\theta(z,\tau)$ converges normally, one may apply standard results in complex analysis \cite[IV.1.7]{b-FreiBus09} to obtain series expansions
\begin{align*}
  \frac{\theta_{0,0}'}{\theta_{0,0}}(z,\tau)&= 2\pi i\sum_{m\in
    \N_0}\Big(\frac{-q^{2m+1}y^{-1}}{1+q^{2m+1}y^{-1}} + \frac{q^{2m+1}y}{1+q^{2m+1}y}\Big),\\
  \frac{\theta_{0,1}'}{\theta_{0,1}}(z,\tau)&= 2 \pi i\sum_{m\in \N_0} \Big(\frac{q^{2m+1}y^{-1}}{1-q^{2m+1}y^{-1}} - \frac{q^{2m+1}y}{1-q^{2m+1}y}\Big),\\
  \frac{\theta_{1,0}'}{\theta_{1,0}}(z,\tau)&= \pi i + 2 \pi i \frac{-y^{-1}}{1+y^{-1}}+ 2 \pi i\sum_{m\in
    \N}\Big(\frac{-q^{2m}y^{-1}}{1+q^{2m}y^{-1}} + \frac{q^{2m}y}{1+q^{2m}y}\Big),\\ 
  \frac{\theta_{1,1}'}{\theta_{1,1}}(z,\tau)&= \pi i + 2 \pi i \frac{y^{-1}}{1-y^{-1}}+ 2 \pi i\sum_{m\in
    \N}\Big(\frac{q^{2m}y^{-1}}{1-q^{2m}y^{-1}} - \frac{q^{2m}y}{1-q^{2m}y}\Big),  
\end{align*} 
which are normally convergent in the domain where $\theta_{i,j}$ is nonzero. Hence
\begin{align}
  \wp_{0,0}(z,\tau)&=-4\pi^2\sum_{m\in
    \N_0}\Big(\frac{q^{2m+1}y^{-1}}{(1+q^{2m+1}y^{-1})^2} +
  \frac{q^{2m+1}y}{(1+q^{2m+1}y)^2}\Big),\label{eq:p00}\\ 
  \wp_{0,1}(z,\tau)&=4\pi^2\sum_{m\in
    \N_0}\Big(\frac{q^{2m+1}y^{-1}}{(1-q^{2m+1}y^{-1})^2} +
  \frac{q^{2m+1}y}{(1-q^{2m+1}y)^2}\Big),\label{eq:p01}\\ 
  \wp_{1,0}(z,\tau)&=-4\pi^2\frac{y^{-1}}{(1+y^{-1})^2}-4\pi^2\sum_{m\in
    \N}\Big(\frac{q^{2m}y^{-1}}{(1+q^{2m}y^{-1})^2} +\frac{q^{2m}y}{(1+q^{2m}y)^2}\Big),\label{eq:p10}\\ 
  \wp_{1,1}(z,\tau)&=4\pi^2 \frac{y^{-1}}{(1-y^{-1})^2}+ 4\pi^2\sum_{m\in
    \N}\Big(\frac{q^{2m}y^{-1}}{(1-q^{2m}y^{-1})^2} + \frac{q^{2m}y}{(1-q^{2m}y)^2}\Big). \label{eq:p11} 
\end{align}
On the other hand, by construction and \eqref{eq:theta11-p}, we get
\begin{align}
  \wp_{0,0}(z,\tau)&=\wp_{1,1}(z+(1+\tau)/2,\tau) = -\wp(z+(1+\tau)/2, \tau)+\eta_1(\tau),\label{eq:P00}\\ 
  \wp_{0,1}(z,\tau)&=\wp_{1,1}(z+\tau/2,\tau)= -\wp(z+\tau/2, \tau) +\eta_1(\tau),\label{eq:P01}\\ 
  \wp_{1,0}(z,\tau)&=\wp_{1,1}(z+1/2,\tau)=-\wp(z+1/2,\tau) + \eta_1(\tau).\label{eq:P10}   
\end{align}
Hence, the functions $\wp_{i,j}(z,\tau)$ for $i,j\in \{0,1\}$ are simultaneously holomorphic for instance in the region 
\begin{align*}
  D=\{(z,\tau)\in \C\times \cH\,|\,-\mathrm{Im}(\tau)/2<\mathrm{Im}(z)<0\}. 
\end{align*}  
Moreover $\eta_1(\tau)$ is a holomorphic function. 
If we want to consider the Weierstrass $\wp$-functions as sums depending on formal variables $y$ and $q$ then we write
\begin{align}\label{eq:wp-formal}
    \wp_{0,0}(y,q), \qquad \wp_{0,1}(y,q), \qquad \wp_{1,0}(y,q), \qquad \wp_{1,1}(y,q).
\end{align}
We will use this notation in particular when $y$ and $q$ are of the form $e^\alpha\in \C[Q]$ for some $\alpha\in Q$. For example $\wp_{i,j}(e^{\alpha_1},e^{-\delta})$ can be considered as elements in the completion $\overline{\cR}_2[P]$ defined in Section \ref{sec:frac-completion}, see Remark \ref{rems:RP}.7.

With the above convention one moreover obtains from \eqref{eq:p00}--\eqref{eq:p11} the relations
\begin{align}
  4\wp_{0,1}(y^2,q^2)&=  \wp_{0,1}(y,q) + \wp_{0,0}(y,q), \label{eq:wp01x22}\\
  4\wp_{1,1}(y^2,q^2)&=  \wp_{1,1}(y,q) + \wp_{1,0}(y,q),  \label{eq:wp11x22}\\
   \wp_{1,1}(y^2,q)  &=  \wp_{0,1}(y^2,q^2) + \wp_{1,1}(y^2,q^2).  \label{eq:wp11x21}
\end{align}
\providecommand{\bysame}{\leavevmode\hbox to3em{\hrulefill}\thinspace}
\providecommand{\MR}{\relax\ifhmode\unskip\space\fi MR }
\providecommand{\MRhref}[2]{%
  \href{http://www.ams.org/mathscinet-getitem?mr=#1}{#2}
}
\providecommand{\href}[2]{#2}


\begin{thebibliography}{BBMR95}

\bibitem[Bas05]{a-Baseilhac05}
P.~Baseilhac, \emph{An integrable structure related with tridiagonal algebras},
  Nuclear Phys. B \textbf{705} (2005), no.~3, 605--619.

\bibitem[BBMR95]{a-BBBG95}
V.~{Beck-Valente}, N.~{Bardy-Panse}, H.~Ben Messaoud, and G.~Rousseau,
  \emph{Formes presque-d{\'e}ploy{\'e}es des alg{\`e}bres de {K}ac-{M}oody:
  {C}lassification et racines relatives}, J. Algebra \textbf{171} (1995),
  43--96.

\bibitem[Ber88]{a-Bernard88}
D.~Bernard, \emph{On the {W}ess-{Z}umino-{W}itten models on the torus}, Nuclear
  Phys. B \textbf{303} (1988), no.~1, 77--93.

\bibitem[Che95a]{a-Cherednik95}
I.~Cherednik, \emph{Difference-elliptic operators and root systems}, Internat.
  Math. Res. Notices \textbf{1995} (1995), no.~1, 43--58.

\bibitem[Che95b]{a-Cherednik95b}
\bysame, \emph{Elliptic quantum many-body problem and double affine
  {K}nizhnik-{Z}amolodohikov equation}, Commun. Math. Phys. \textbf{169}
  (1995), 441--461.

\bibitem[CM82]{a-CM82}
W.~Casselman and D.~Mili{\v c}i{\'c}, \emph{Asymptotic behavior of matrix
  coefficients of admissible representations}, Duke Math. J. \textbf{49}
  (1982), no.~4, 869--930.

\bibitem[Dar82]{a-Darboux1882}
G.~Darboux, \emph{Sur une \'equation lin\'eaire}, C. R. Acad Sci. Paris
  \textbf{XCIV} (1882), no.~25, 1645--1648.

\bibitem[EFK95]{a-EFK95}
P.I. Etingof, I.B. Frenkel, and A.A. {Kirillov, Jr.}, \emph{Spherical functions
  on affine {L}ie groups}, Duke Math. J. \textbf{80} (1995), no.~1, 59--90.

\bibitem[EK94]{a-EK94}
P.I. Etingof and A.A. {Kirillov, Jr.}, \emph{Representations of affine {L}ie
  algebras, parabolic differential equations, and {L}am\'e functions}, Duke
  Math. J. \textbf{74} (1994), no.~3, 585--614.

\bibitem[EK95]{a-EK95}
\bysame, \emph{On the affine analogue of {J}ack and {M}acdonald polynomials},
  Duke Math. J. \textbf{78} (1995), no.~2, 229--256.

\bibitem[FB09]{b-FreiBus09}
E.~Freitag and R.~Busam, \emph{Complex analysis. {S}econd edition},
  Springer-Verlag, Berlin, 2009.

\bibitem[FW96]{a-FelderWiecz96}
G.~Felder and C.~Wieczerkowski, \emph{Conformal blocks on elliptic curves and
  the {K}nizhnik-{Z}amolodchikov-{B}ernard equations}, Commun. Math. Phys.
  \textbf{176} (1996), no.~1, 133--161.

\bibitem[Gel50]{a-Gelfand50}
I.M. Gelfand, \emph{Spherical functions in symmetric {R}iemann spaces}, Doklady
  Akad. Nauk SSSR (N.S.) \textbf{70} (1950), 5--8.

\bibitem[Hay88]{a-Hayashi88}
T.~Hayashi, \emph{Sugawara operators and {K}ac-{K}azhdan conjecture}, Invent.
  Math. \textbf{94} (1988), 13--52.

\bibitem[HC58]{a-HarishChandra58}
Harish-Chandra, \emph{Spherical functions on a semisimple {L}ie group. {I}},
  Amer. J. Math. \textbf{80} (1958), 241--310.

\bibitem[HO87]{a-HeckmanOpdam87}
G.~Heckman and E.~Opdam, \emph{Root systems and hypergeometric functions. {I}},
  Compositio Math. \textbf{64} (1987), no.~3, 329--352.

\bibitem[HS94]{b-HeckSchlicht94}
G.~Heckman and H.~Schlichtkrull, \emph{Harmonic analysis and special functions
  on symmetric spaces}, Perspectives in Mathematics, no.~16, Academic Press,
  San Diego, 1994.

\bibitem[Ino89]{a-Ino89}
V.I. Inozemtsev, \emph{Lax representation with spectral parameter on a torus
  for integrable particle systems}, Lett. Math. Phys. \textbf{17} (1989),
  11--17.

\bibitem[Kac90]{b-Kac1}
V.~G. Kac, \emph{Infinite dimensional {L}ie algebras}, 3rd. ed., Cambridge
  University Press, Cambridge, 1990.

\bibitem[{Kir}95]{phd-Kirillov95}
A.~A. {Kirillov, Jr.}, \emph{Traces of intertwining operators and {M}acdonald's
  polynomials}, Ph.D. thesis, Yale University, 1995, {\ttfamily
  arXiv:q-alg/9503012}.

\bibitem[Kol12]{a-Kolb12p}
S.~Kolb, \emph{Quantum symmetric {K}ac-{M}oody pairs}, preprint, {\ttfamily
  arXiv:1207.6306} (2012), 61 pp.

\bibitem[KP83]{a-KP83}
V.G. Kac and D.H. Peterson, \emph{Regular functions on certain
  infinite-dimensional groups}, Arithmetic and Geometry (M.~Artin and J.~Tate,
  eds.), Progress in Math., vol.~36, Birkh{\"a}user, Boston, 1983,
  pp.~141--166.

\bibitem[KW92]{a-KW92}
V.G. Kac and S.P. Wang, \emph{On automorphisms of {K}ac-{M}oody algebras and
  groups}, Adv. Math. \textbf{92} (1992), 129--195.

\bibitem[Let00]{a-Letzter00}
G.~Letzter, \emph{Harish-{C}handra modules for quantum symmetric pairs},
  Representation Theory, An Electronic Journal of the AMS \textbf{4} (2000),
  64--96.

\bibitem[Let04]{a-Letzter04}
\bysame, \emph{Quantum zonal spherical functions and {M}acdonald polynomials},
  Adv. Math. \textbf{189} (2004), 88--147.

\bibitem[Loo76]{a-Looijenga76}
E.~Looijenga, \emph{Root systems and elliptic curves}, Invent. Math.
  \textbf{38} (1976), 17--32.

\bibitem[Loo80]{a-Looijenga80}
\bysame, \emph{Invariant theory for generalized root systems}, Invent. Math.
  \textbf{61} (1980), 1--32.

\bibitem[LT12]{a-LangTak12}
E.~Langmann and K.~Takemura, \emph{Source identity and kernel functions for
  {I}nozemtsev-type systems}, J. Math. Phys. \textbf{53} (2012), no.~8, 082105,
  19 pp.

\bibitem[Mok03]{a-Mokler03}
C.~Mokler, \emph{A formal {C}hevalley restriction theorem for {K}ac-{M}oody
  groups}, Compositio Math. \textbf{135} (2003), 123--152.

\bibitem[MR88]{b-McCoRo}
J.C. McConnell and L.C. Robson, \emph{Noncommutative noetherian rings}, John
  Wiley \& Sons, New York, 1988.

\bibitem[MS06]{a-MatSmir06}
V.B. Matveev and A.O. Smirnov, \emph{On the link between the {S}parre equation
  and {D}arboux-{T}reibich-{V}erdier equation}, Lett. Math. Phys. \textbf{76}
  (2006), 283--295.

\bibitem[Mum83]{b-Mum83}
D.~Mumford, \emph{Tata lectures on theta {I}}, Progress in Mathematics, no.~28,
  Birkh{\"a}user, Boston, 1983.

\bibitem[Nou96]{a-Noumi96}
M.~Noumi, \emph{Macdonald's symmetric polynomials as zonal spherical functions
  on some quantum homogeneous spaces}, Adv. Math. \textbf{123} (1996), 16--77.

\bibitem[Ons44]{a-Onsager44}
L.~Onsager, \emph{Crystal statistics. {I}. {A} two-dimensional model with an
  order-disorder transition}, Phys. Rev. (2) \textbf{65} (1944), 117--149.

\bibitem[OOS94]{a-OOS94}
H.~Ochiai, T.~Oshima, and H.~Sekiguchi, \emph{Commuting families of symmetric
  differential operators}, Proc. Japan Acad. Ser. A Math Sci. \textbf{70}
  (1994), no.~2, 62--66.

\bibitem[OP83]{a-OlshPer83}
M.A. Olshanetsky and A.M. Perelomov, \emph{Quantum integrable systems related
  to {L}ie algebras}, Rev. Math. Phys. \textbf{94} (1983), no.~6, 313--404.

\bibitem[Opd88]{a-Opdam88}
E.~Opdam, \emph{Root systems and hypergeometric functions. {I}{V}}, Compositio
  Math. \textbf{67} (1988), no.~2, 191--209.

\bibitem[OS95]{a-OS95}
T.~Oshima and H.~Sekiguchi, \emph{Commuting families of differential operators
  invariant under the action of a {W}eyl group}, J. Math. Sci. Univ. Tokyo
  \textbf{2} (1995), 1--75.

\bibitem[Pol03]{b-Polishchuk03}
A.~Polishchuk, \emph{Abelian varieties, theta functions and the {F}ourier
  transform}, Cambridge University Press, Cambridge, 2003.

\bibitem[Ros]{a-Rosengren13p}
H.~Rosengren, \emph{Special polynomials related to the supersymmetric
  eight-vertex model. {I}{I}. {S}chr{\"o}dinger equation}, in preparation.

\bibitem[RS86]{a-RuijSchneider86}
S.~Ruijsenaars and H.~Schneider, \emph{A new class of integrable systems and
  its relation to solitons}, Ann. Physics \textbf{170} (1986), no.~2, 370--405.

\bibitem[Rui87]{a-Ruijsenaars87}
S.~Ruijsenaars, \emph{Complete integrability of relativistic {C}alogero-{M}oser
  systems and elliptic function identities}, Commun. Math. Phys. \textbf{110}
  (1987), no.~2, 191--213.

\bibitem[Rui09]{a-Ruijsenaars09}
\bysame, \emph{Hilbert-{S}chmidt operators vs. integrable systems of elliptic
  {C}alogero-{M}oser type. {I}{I}{I}. {T}he {H}eun case.}, SIGMA Symmetry
  Integrability Geom. Methods Appl. \textbf{5} (2009), paper 049, 21 pp.

\bibitem[Tak03]{a-Takemura03}
K.~Takemura, \emph{The {H}eun equation and the {C}alogero-{M}oser-{S}utherland
  system.{I}. {T}he {B}ethe {A}nsatz method}, Commun. Math. Phys. \textbf{235}
  (2003), no.~3, 467 -- 494.

\bibitem[Var03]{b-Varchenko03}
A.~Varchenko, \emph{Special functions, {KZ} type equations, and representation
  theory}, CBMS Regional Conference Series in Mathematics, no.~98, AMS,
  Providence, RI, 2003.

\bibitem[vD94]{a-vanDiejen94}
J.F. van Diejen, \emph{Integrability of difference {C}alogero-{M}oser systems},
  J. Math. Phys. \textbf{35} (1994), no.~6, 2983--3004.

\bibitem[Ves11]{a-Veselov11}
A.P. Veselov, \emph{On {D}arboux-{T}reibich-{V}erdier potentials}, Lett. Math.
  Phys. \textbf{96} (2011), 209--216.

\end{thebibliography}
\end{document}